\documentclass[12pt]{amsart}
\usepackage[utf8]{inputenc}
\usepackage[left=2.cm,right=2.cm,top=3cm,bottom=2cm]{geometry}
\usepackage{graphicx}						
\usepackage{amssymb}
\usepackage{graphicx}
\usepackage{epstopdf}
\usepackage{color}
\usepackage{pgf,tikz}
\usetikzlibrary{arrows}
\usepackage[colorlinks=true,citecolor=cyan,urlcolor=blue]{hyperref} 
\usepackage{bm}

\newtheorem{theo}{Theorem}
\newtheorem*{theo*}{Theorem}
\newtheorem{cl}{Claim}
\newtheorem{lem}{Lemma}
\newtheorem*{lem*}{Lemma}

\newtheorem{prop}{Proposition}

\newcommand{\SST}{\mathop{\mathrm{SST}}}
\newcommand{\SYT}{\mathop{\mathrm{SYT}}}

\newcommand{\A}{\mathop{\mathrm{area}}}

\newcommand{\h}{\mathop{\mathrm{ht}}}

\newcommand{\dd}{\mathop{\mathrm{des}}}
\newcommand{\DD}{\mathop{\mathrm{Des}}}
\newcommand{\m}{\mathop{\mathrm{maj}}}
\newcommand{\M}{\mathop{\mathrm{Des}}}
\newcommand{\hooks}{\mathop{\mathrm{hooks}}}
\newcommand{\hook}{\mathop{\mathrm{hook}}}
\newcommand{\shape}{\mathop{\mathrm{Shape}}}
 
\newcommand{\pp}{\mathop{\mathrm{Part}}}

\title[Explicit formulas for Characters of the space of multivariate diagonal harmonics]{Explicit Combinatorial Formulas for  Some Irreducible Characters of the $GL_k\times \mathbb{S}_n$-module of multivariate diagonal harmonics}

\author{\href{wallace.nancy@courrier.uqam.ca}{Nancy Wallace}} 

\begin{document}
\maketitle
\tableofcontents
\footnote{This work is supported by a scholarship from the NSERC.}

\begin{abstract}We give an explicit combinatorial formula for some irreducible components of $GL_k\times \mathbb{S}_n$-modules of multivariate diagonal harmonics. To this end we introduce a new path combinatorial object $T_{n,s}$ allowing us to give the formula directly in terms of Schur functions. This paper also contains formulas written in terms of Schur functions in the $q$ and $t$ variables for special cases of $\nabla(e_n)$, $\nabla^r(e_n)$ and $\Delta'_{e_k}(e_n)$. We also give an interpretation in term of path to the adjoint dual Pieri rule applied on these $GL_k\times \mathbb{S}_n$-characters.
\end{abstract}


\section{Introduction}
The aim of his paper is to describe some features of the characters of  ``rectangular" $GL_k\times\mathbb{S}_n$-modules, $\mathcal{E}_{m,n}^{\langle k\rangle}$, introduced by F. Bergeron in \cite{[B2018]}.

When $k=2$ and $m=n$, these modules are the modules of diagonal harmonics whose characters have been studied for many years. As shown in \cite{[GH1996a]} and \cite{[H2002]}, the Frobenius transformation of its graded characters may be expressed as $\nabla(e_n)$  where $\nabla$ is the Macdonald eigenoperator introduced in \cite{[BG1999]} and recalled in Section \ref{Sec : car} and $e_n$ is the $n$-th elementary symmetric function. A combinatorial interpretation that became known as the Shuffle Conjecture was introduced in \cite{[HHLRU2005]} and proved recently by Carlson and Mellit \cite{[CM2015]}, \cite{[M2016]}. These characters also intervene in torus knot link homology and algebraic geometry; see, for instance \cite{[H2017]}, \cite{[K2007]},\cite{[GNR2016]} and \cite{[OR2018]}. The case $k=3$ was studied in \cite{[BP2012]}. 

In recent work \cite{[B2018]} F. Bergeron made a breakthrough in the multivariate case ($k$ arbitrary). He found interesting relations between various irreducible characters of the modules. This allowed the study of the character of $\mathcal{E}_{n,n}^{\langle k\rangle}$ developed in the elementary symmetric functions for $n\leq 4$ in \cite{[BCP2018]} (in our notation this is the specialization $\sum c_{\lambda,\mu}s_\lambda(q,1,\ldots,1,0,0,\ldots)s_\mu(X)$, where there are $k-1$, $1$'s). These relations are also exploited here to obtain our main result. 

To state our result we briefly fix the required notation. We encode the characters of the irreducible $GL_k\times \mathbb{S}_n$ -modules  as products of Schur functions, $s_\lambda(\bm{Q})s_\mu(\bm{X})$ (see Section \ref{Sec : car} for details). For easier reading we will also write  $s_\lambda\otimes s_\mu$ for $s_\lambda(\bm{Q})s_\mu(\bm{X})$. As shown in \cite{[B2013]} for the case $m=n$ there is a stability property that makes it possible to avoid mentioning $k$. Therefore, the character of $\mathcal{E}_{m,n}^{\langle k \rangle}$ can be expressed in the form $\mathcal{E}_{m,n}=\sum_\mu \left(\sum_\lambda c_{\lambda,\mu}s_\lambda\right)\otimes s_\mu$. Our aim is to describe some features of $\mathcal{E}_{m,n}$, or equivalently, $\langle \mathcal{E}_{m,n},s_\mu\rangle =\sum_\lambda c_{\lambda,\mu}s_\lambda$. 

The main result of this paper is a combinatorial description of the multiplicity of $s_\lambda\otimes s_\mu$ in $\mathcal{E}_{n,
  n}$ when $\lambda$ is hook shaped. The result is constructive, in that the hooks are determined combinatorially by a standard Young tableau and certain paths in a staircase shaped grid. More precisely, we give the following combinatorial description (the basic combinatorial notations used in Theorem \ref{Thm : main} are recalled in Section \ref{Sec : tools}).

  \begin{theo}\label{Thm : main}If  $r=1$ and $\mu\in\{(n), (n-1,1), (n-2,1,1),(1^n)\}$ then:
  \begin{equation}\label{Eq : main}\langle \mathcal{E}_{rn,n}, s_{\mu}\rangle|_{\hooks}=\sum_{\tau\in \SYT(\mu)} \sum_{\gamma \in T_{n, \dd(\tau')}} s_{\hook(\gamma)},
  \end{equation}
where the first sum is over all standard Young tableaux, $\tau$, of shape $\mu$, the second sum is over paths $\gamma$ in $T_{n, \dd(\tau')}$ and $\hook(\gamma)=\left((r-1)\binom{n}{2} +\A(\gamma)+\h(\gamma)-\m(\tau')+1,1^{n-2-\h(\gamma) }\right)$.

Furthermore, when $\mu=1^n$, Equation \eqref{Eq : main} holds for all positive integers $r$ that satisfy the equality $ \langle\mathcal{E}_{rn,n}, s_{k+1,1^{n-k-1}}\rangle=e^\perp_k \langle\mathcal{E}_{rn,n}, s_{1^{n}}\rangle$ for all $k$.
\end{theo}

In addition, Proposition \ref{Prop : perp sur chemin} gives a lower bound for the coefficients, which gives reason to believe that Equation \eqref{Eq : main} works for all $\mu$ when $r=1$. Moreover, in \cite{[B2018]} it is conjectured that $e_k^\perp \langle \mathcal{E}_{n,n}, s_\mu\rangle=_{\langle 2\rangle} \langle\Delta'_{e_{n-k-1}}e_n, s_\mu\rangle$ is true for all $k$. The notation $=_{\langle 2\rangle}$ denotes the restriction $\bm{Q}=\{q,t\}$. If this conjecture is true then  Equation \eqref{Eq : main} holds for all $\mu$ when $r=1$.

The combinatorial object $T_{n,s}$ ($T_{n,\dd(\tau')}$ in Equation \eqref{Eq : main}) represents a set of paths, these paths and their statistics ($\A$ and $\h$) will be defined in Section \ref{Sec : new object}. This object was introduced by the author in order to eliminate alternating sums and obtain a Schur positive expression. When $s=0$, they afford a generating function correlated to the $q$-Pochhammer symbol $(-qz;z)_n$. The notation $|_{\hooks}$, symmetric functions $s_\lambda$, $e_n$, the operators $\nabla$, $\Delta'$ and $e_k^\perp$ will be recollected in Section \ref{Sec : car}. Section \ref{Sec : main} will be dedicated to Equation \eqref{Eq : main},  and a proposition restricting Theorem \ref{Thm : main} to the $GL_2\times \mathbb{S}_n$-characters mentioned above, which gives formulas in terms of the major index for $\langle \nabla^r(e_n), s_{\mu}\rangle|_{\hooks}$, $\langle \Delta'_{e_k}(e_n),e_n\rangle|_{\hooks}$ and $\langle \Delta'_{e_k}(e_n),s_\mu\rangle|_{\textrm{some hooks}}$. In Section \ref{Sec : Pieri} we show how the adjoint dual Pieri rule can be applied directly on our paths. Finally Section \ref{Sec : 2e colonne} gives a formula similar to Equation \eqref{Eq : main} for shapes $\{(a,2,1^k)~|~ k\in\mathbb{N}, a\in \mathbb{N}_{\geq 2}\}$.

  
  \section{Combinatorial tools}\label{Sec : tools}

In this section the notions are classical, they are recalled here only to set notations.
A \begin{it}partition\end{it} of $n$ is a decreasing sequence of positive integers often represented as a Ferrer's diagram (see Figure \ref{Fig : tableau}). Each number in the sequence is called a \begin{it}part\end{it} and if it has $k$ parts it is of length $k$ denoted $\ell(\lambda)=k$. We say $\lambda$ is of \begin{it}size\end{it} $n$ if $\lambda=\lambda_1,\ldots,\lambda_k$ and $n=\sum_i\lambda_i$. For $\lambda$ a  \begin{it}Ferrer's diagram \end{it} of shape $\lambda=\lambda_1,\lambda_2,\ldots,\lambda_k$ is a left justified pile of boxes having $\lambda_i$ boxes in the $i$-th row. We will use the French notation so the second row lies on top of the first row (see Figure \ref{Fig : tableau}). We can see them as a subset of $\mathbb{N}\times\mathbb{N}$ if we put the bottom left corner of the diagram to the origin. In this setting, we can associate the bottom left corner of a box to the coordinate it lies on. We say a partition is  \begin{it}hook shaped\end{it} if it has the form $(a,1, \ldots,1)=(a,1^k)$, where $a, k \in \mathbb{N}$. If to each box of a diagram we associate an entry it is called a \begin{it}tableau\end{it}. A tableau is of shape $\lambda$ if it is a filling by integers of a diagram of shape $\lambda$. For $\lambda$ a partition of $n$, a tableau of shape $\lambda$ with distinct entries from $1$ to $n$ strictly increasing in rows and columns is a  \begin{it}standard Young tableau\end{it}. It is a \begin{it}semi-standard Young tableau\end{it} if the row entries are weakly increasing and the columns entries are strictly increasing. The set of all standard Young tableaux of shape $\mu$ is denoted $\SYT(\mu)$ and the set of all semi-standard Young tableaux of shape $\mu$ is denoted $\SST(\mu)$. The  \begin{it}descent\end{it} set of a tableau is the set of entries $i$ such that the entry $i+1$ lies in a row strictly above $i$. For a tableau $\tau$, the descent set is denoted $\DD(\tau)$ and the cardinality is denoted  $\dd(\tau)$. The sum of the elements of  $\DD(\tau)$ is called the  \begin{it}major index\end{it} and is denoted $\m(\tau)$ (see Figure \ref{Fig : descents}). The  \begin{it}conjugate \end{it} of a partition $\lambda$, (respectively a diagram $\lambda$, a tableau $\tau$) is denoted $\lambda'$ (respectively a diagram $\lambda'$, a tableau $\tau'$), and is its reflection through the line $x=y$ (see Figure \ref{Fig : conjugate}).

\begin{figure}[h!]
\begin{minipage}{4.7cm}
\centering
\begin{tikzpicture}[scale=.5]
\draw (0,0)--(4,0)--(4,1)--(2,1)--(2,3)--(1,3)--(1,5)--(0,5)--(0,0);
\draw (1,0)--(1,3);
\draw(2,0)--(2,1);
\draw (3,0)--(3,1);
\draw (4,0)--(4,1);
\draw (0,1)--(2,1);
\draw (0,2)--(2,2);
\draw (0,3)--(1,3);
\draw (0,4)--(1,4);
\end{tikzpicture}
\caption{
$\lambda=42211$}
\label{Fig : tableau}
\end{minipage}
\begin{minipage}{4.7cm}
\centering
\begin{tikzpicture}[scale=.5]
\draw (0,0)--(5,0)--(5,1)--(3,1)--(3,2)--(1,2)--(1,4)--(0,4)--(0,0);
\draw (1,0)--(1,2);
\draw(2,0)--(2,2);
\draw (3,0)--(3,1);
\draw (4,0)--(4,1);
\draw (0,1)--(3,1);
\draw (0,2)--(1,2);
\draw (0,3)--(1,3);
\draw (0,4)--(1,4);
\end{tikzpicture}
\caption{
$\lambda'=5311$}
\label{Fig : conjugate}
\end{minipage}
\begin{minipage}{7.8cm}
\centering
\begin{tikzpicture}[scale=.5]
\draw (0,0)--(4,0)--(4,1)--(2,1)--(2,3)--(1,3)--(1,5)--(0,5)--(0,0);
\draw (1,0)--(1,3);
\draw(2,0)--(2,1);
\draw (3,0)--(3,1);
\draw (4,0)--(4,1);
\draw (0,1)--(2,1);
\draw (0,2)--(2,2);
\draw (0,3)--(1,3);
\draw (0,4)--(1,4);
\node(1) at (.5,.5){$\textcolor{black}{1}$};
\node(2) at (1.5,.5){$\textcolor{black}{2}$};
\node(3) at (.5,1.5){$3$};
\node(4) at (2.5,.5){$\textcolor{black}{4}$};
\node(5) at (.5,2.5){$\textcolor{black}{5}$};
\node(6) at (.5,3.5){$6$};
\node(7) at (1.5,1.5){$7$};
\node(8) at (3.5,.5){$8$};
\node(9) at (.5,4.5){$\textcolor{black}{9}$};
\node(10) at (1.5,2.5){$\textcolor{black}{10}$};
\end{tikzpicture}
\caption{$\tau\in \SYT(42211)$;
$\dd(\tau)=4$, $\m(\tau)=19$
$\DD(\tau)=\{2,4,5,8\}$,}
\label{Fig : descents}
\end{minipage}
\end{figure}

We end this section by recalling the definition of the Gaussian polynomials, $\left[ \genfrac{}{}{0pt}{}{n}{k}\right]_q$. Let:
\begin{equation*}[n]_q=1+q+q^2+\cdots+q^{n-1} \text{ , }~[n]!_q=\prod_{i=1}^n [i]_q
\text{ and }\begin{bmatrix} n\\k\end{bmatrix}_q=\frac{[n]!_q}{[n-k]!_q[k]!_q}.
\end{equation*}
So when $q=1$ this gives the usual binomial coefficient. It is well known that the Gaussian polynomials are related to the north-east paths of a $k \times (n-k)$ grid: if $\mathcal{C}_{k}^n$ denotes the set of such paths, and the \begin{it}area\end{it} of a path $\gamma$, denoted $\A(\gamma)$, is defined as the number of boxes beneath the path, then $\sum_{\gamma\in\mathcal{C}_k^n} q^{area(\gamma)}=\left[ \genfrac{}{}{0pt}{}{n}{k}\right]_q$. 

 
   \section{The Space, the Characters, symmetric functions and Macdonald operators}\label{Sec : car}
   
    The symmetric function notations are the one used in Macdonald's book \cite{[Mac1995]}. The \begin{it}ring of symmetric polynomials\end{it} is a set of polynomials which are invariant by permutation of the variables $Y=\{y_1,y_2, \ldots, y_n\}$. The ring of symmetric polynomials is embedded in $\mathbb{Q}[Y]$. In other words if $f$ is a symmetric polynomials for all $\sigma \in \mathbb{S}_n$ we have:
\begin{equation*}
f(y_1,y_2,\ldots, y_n)=f(y_{\sigma^{-1}(1)}, y_{\sigma^{-1}(2)}, \ldots, y_{\sigma^{-1}(n)}).
\end{equation*}  
 The \begin{it}ring of symmetric functions\end{it}, denoted $\Lambda$, can be thought of as the ring of symmetric polynomials in an infinite set of variables. It is a graded ring and has the \begin{it}elementary symmetric functions\end{it} as a basis. These are defined by $e_n(X)=\sum_{i_1<\cdots<i_n} x_{i_1}x_{i_2}\cdots x_{i_n}$ and $e_\lambda=e_{\lambda_1}e_{\lambda_2}\cdots e_{\lambda_k}$. Schur functions also form a basis for the ring of symmetric functions. We now recall the combinatorial definition of Schur functions. For a tableau $\tau$ we define $x_\tau:=\prod_{c\in\tau}x_c$. The Schur functions are then defined by $s_\lambda=\sum_{\tau\in \SST(\lambda)} x_\tau$. For example, if $x_1=x$, $x_2=y$ and $x_3=z$, we have:
\begin{align*}s_{21}(x,y,z,\ldots)&=\begin{tikzpicture}[line cap=round,line join=round,>=triangle 45,x=1cm,y=1cm, thick, every node/.style={scale=0.7},scale=.3]
\draw (0,3)--(0,1)--(2,1)--(2,2);
\draw (0,1)--(2,1);
\draw (0,2)--(2,2);
\draw (0,3)--(1,3)--(1,1);
\draw (0.5,1.5) node{1};
\draw (0.5,2.5) node{2};
\draw (1.5,1.5) node{1};
\end{tikzpicture}+\begin{tikzpicture}[line cap=round,line join=round,>=triangle 45,x=1cm,y=1cm, thick, every node/.style={scale=0.7},scale=.3]
\draw (0,3)--(0,1)--(2,1)--(2,2);
\draw (0,1)--(2,1);
\draw (0,2)--(2,2);
\draw (0,3)--(1,3)--(1,1);
\draw (0.5,1.5) node{1};
\draw (0.5,2.5) node{3};
\draw (1.5,1.5) node{1};
\end{tikzpicture}+\begin{tikzpicture}[line cap=round,line join=round,>=triangle 45,x=1cm,y=1cm, thick, every node/.style={scale=0.7},scale=.3]
\draw (0,3)--(0,1)--(2,1)--(2,2);
\draw (0,1)--(2,1);
\draw (0,2)--(2,2);
\draw (0,3)--(1,3)--(1,1);
\draw (0.5,1.5) node{1};
\draw (0.5,2.5) node{2};
\draw (1.5,1.5) node{2};
\end{tikzpicture}+\begin{tikzpicture}[line cap=round,line join=round,>=triangle 45,x=1cm,y=1cm, thick, every node/.style={scale=0.7},scale=.3]
\draw (0,3)--(0,1)--(2,1)--(2,2);
\draw (0,1)--(2,1);
\draw (0,2)--(2,2);
\draw (0,3)--(1,3)--(1,1);
\draw (0.5,1.5) node{1};
\draw (0.5,2.5) node{3};
\draw (1.5,1.5) node{3};
\end{tikzpicture}+\begin{tikzpicture}[line cap=round,line join=round,>=triangle 45,x=1cm,y=1cm, thick, every node/.style={scale=0.7},scale=.3]
\draw (0,3)--(0,1)--(2,1)--(2,2);
\draw (0,1)--(2,1);
\draw (0,2)--(2,2);
\draw (0,3)--(1,3)--(1,1);
\draw (0.5,1.5) node{2};
\draw (0.5,2.5) node{3};
\draw (1.5,1.5) node{2};
\end{tikzpicture}+\begin{tikzpicture}[line cap=round,line join=round,>=triangle 45,x=1cm,y=1cm, thick, every node/.style={scale=0.7},scale=.3]
\draw (0,3)--(0,1)--(2,1)--(2,2);
\draw (0,1)--(2,1);
\draw (0,2)--(2,2);
\draw (0,3)--(1,3)--(1,1);
\draw (0.5,1.5) node{2};
\draw (0.5,2.5) node{3};
\draw (1.5,1.5) node{3};
\end{tikzpicture}+\begin{tikzpicture}[line cap=round,line join=round,>=triangle 45,x=1cm,y=1cm, thick, every node/.style={scale=0.7},scale=.3]
\draw (0,3)--(0,1)--(2,1)--(2,2);
\draw (0,1)--(2,1);
\draw (0,2)--(2,2);
\draw (0,3)--(1,3)--(1,1);
\draw (0.5,1.5) node{1};
\draw (0.5,2.5) node{3};
\draw (1.5,1.5) node{2};
\end{tikzpicture}+\begin{tikzpicture}[line cap=round,line join=round,>=triangle 45,x=1cm,y=1cm, thick, every node/.style={scale=0.7},scale=.3]
\draw (0,3)--(0,1)--(2,1)--(2,2);
\draw (0,1)--(2,1);
\draw (0,2)--(2,2);
\draw (0,3)--(1,3)--(1,1);
\draw (0.5,1.5) node{1};
\draw (0.5,2.5) node{2};
\draw (1.5,1.5) node{3};
\end{tikzpicture}+\cdots
\\ &=x^2y+x^2z+xy^2+xz^2+y^2z+yz^2+2xyz+\cdots
\end{align*}

Note that the columns are strictly increasing therefore we need a number of variables greater or equal to the length of the partition. 
 
 The ring of symmetric functions in the variables $Q=\{q_1,q_2,\ldots\}$ will be noted $\Lambda_Q$, the ring of symmetric functions in the variables $x_1,x_2,\ldots$ will be noted $\Lambda$. The product of a symmetric function $f$ in $\Lambda_Q$ and a symmetric function $g$ in $\Lambda$  will be noted $f\otimes g$. The set of such functions will be noted $\Lambda_Q\otimes \Lambda$. It is easy to see that it is a bigraded ring.
 If an element of $\Lambda$ (respectively $\Lambda_Q$, $\Lambda_Q\otimes \Lambda$) can be written in the basis of Schur function (respectively the basis $\{s_\lambda\}$, the basis $\{s_\lambda\otimes s_\mu\}$) with coefficients in $\mathbb{N}$ or $\mathbb{N}[q,t]$ (respectively $\mathbb{N}$, $\mathbb{N}$) is said to be Schur positive. The $\omega$ linear operator is defined by $\omega(s_\mu(X))=s_{\mu'(X)}$ which extends to $\Lambda_Q\otimes \Lambda$ by $\omega(s_\lambda\otimes s_\mu)=s_\lambda\otimes s_{\mu'}$.
 
  If $\sum_\lambda c_\lambda s_\lambda$ is a linear combination of Schur functions, and $\mathcal{V}$ is a set of shapes, then we define the notation $|_{\mathcal{V}}$ by $\sum_\lambda c_\lambda s_\lambda|_\mathcal{V}=\sum_{\lambda \in \mathcal{V}} c_\lambda s_\lambda$. On a symmetric function in the Schur basis, we set the restriction $|_{1 \textrm{part}}$ (respectively  $|_{\hooks}$) to be the partial sum over the Schur functions indexed by partitions having only one part (respectively that are hook shaped). 
For example if $\mathcal{V}=\{ 1 11, 3 2, 6 \}$ and $f=5 s_{1 11}+ 5s_{3 1}+s_{41}+2q^6s_{6}$ then:
 \begin{equation*}f|_{\mathcal{V}}=5 s_{1 11}+2q^6s_{6}.
 \end{equation*}

Before introducing our new combinatorial objects we provide more details about the modules $\mathcal{E}_{n,n}^{\langle k\rangle}$ and why they are interesting.  

 Let $X=(x_{i,j})_{i,j}$ where $1\leq i\leq k$,  $1\leq j\leq n$ and let $R_n^{\langle k \rangle}=\mathbb{Q}[X]$ denote the polynomial ring in the variables $X$. For $(\tau,\sigma)$ in $GL_k\times \mathbb{S}_n$ the group $GL_k\times \mathbb{S}_n$ acts on $R_n^{\langle k \rangle}$ as follows:  
   \begin{equation*}
    (\tau,\sigma). F(X)=F(\tau \cdot X \cdot \sigma)
    \end{equation*}
With this action we can define $\mathcal{E}_{n,n}^{\langle k \rangle}$ as the smallest submodule of $R_n^{\langle k \rangle}$ that contains the Vandermonde determinant, is closed under all higher polarization operators $\sum_{j=0}^n x_{r,j}\partial_{x_{s,j}}$ and  is closed under all partial derivatives $ \partial_{x_{s,j}}$.

As usual we may decompose this modules into a direct sum of irreducible modules. A module can be encoded by its character. Recall that both the irreducible characters for $GL_k$ and the Frobenius transforms of irreducible characters of $\mathbb{S}_n$ are Schur functions. 

For example, when $n=4$ we have:
\begin{align*}\mathcal{E}_{4,4}=&\newcommand{\Bold}[1]{\mathbf{\otimes1}}1 \otimes s_{4} + (s_{1}+ s_{2}+s_{3}) \otimes s_{31}+ (s_{2}+s_{21}+s_{4}) \otimes s_{22}
\\	&+(s_{11}+s_{2 1}+s_{31}+ s_{3}+ s_{4}+s_{5} ) \otimes s_{211} + (s_{111}+ s_{31}+s_{41}+s_{6}) \otimes s_{1111}
\end{align*}

If we start with the ring $\mathbb{Q}(q,t)[X]$ we also get a ring with Schur functions as a basis. Additionally, the combinatorial Macdonald polynomials, denoted $H_\mu$, are a basis for $\mathbb{Q}(q,t)[X]$. They appear as eigenvectors of special operators (see \cite{[B2009]} for more on this), $\nabla$ and the $\Delta'_{e_m}$, introduced in \cite{[BG1999]}, \cite{[BGHT1999]}. These Macdonald operators are defined as follows:
\begin{equation*}\nabla(H_\mu)=\prod_{(i,j)\in \mu}q^it^jH_\mu~ \text{ and }~\Delta'_{e_m}(H_\mu)=e_m\left[\sum_{(i,j)\in\mu}q^it^j-1\right]H_\mu.
\end{equation*}
The brackets are for plethysm. The notion of plethysm is not needed in this paper, but the curious reader could learn more on this in \cite{[B2009]}. The bivariate diagonal harmonics space was proven to have $\nabla(e_n(X))$ as a character in \cite{[H2002]}. Ergo $\mathcal{E}_{n,n}$ affords the following specialization:
\begin{equation*} \mathcal{E}_{n,n}^{\langle 2\rangle}=\nabla(e_n(X)).
\end{equation*}
One might notice from our previous example that if we take out the term $s_{111}\otimes s_{1111}$, or equivalently, set $q_1=q$, $q_2=t$, $q_j=0$ for $j\geq 3$ we have $\nabla(e_4)$. That extra term isn't a problem since $s_{111}(q,t,0,0,\ldots)=0$. As noted beforehand, in two variables Schur functions vanish if they have more then two parts.

By definition we have $\Delta'_{e_{n-1}}(e_n)=\nabla(e_n)$, which gives the character decomposition of  the $GL_2\times \mathbb{S}_n$ case stated in the introduction. It was proven that the coefficients are symmetric polynomials in the $q$ and $t$ variables, thus one could write the coefficients in the form $\sum_{\lambda,\mu} c_{\lambda,\mu}s_\lambda(q,t,0,\ldots)s_\mu(X)$. More generally, the $GL_k$ characters can be obtained by setting  $q_{k+1}=q_{k+2}=\cdots=0$. A stability property was proven in \cite{[B2013]}, we can therefore set $k$ to infinity and use a more general notation $\mathcal{E}_{n,n}$.

Moreover, F.Bergeron conjectures in \cite{[B2018]} that the restriction to two variables of $e_k^\perp \langle \mathcal{E}_{n,n}, s_\mu\rangle$ is equal to $\langle\Delta'_{e_{n-k-1}}e_n, s_\mu\rangle$ for all $k$. If this conjecture is true the Schur positive development of $\mathcal{E}_{n,n}$ (exists since its a character) gives us the Schur positive development of $\Delta'_{e_{k}}e_n$.

We will also discuss characters of the form $\mathcal{E}_{rn,n}$ which are related to $\nabla^r(e_n)$ when we restrict to $GL_2$. They are constructed by adding a set of inert variables considered to be of degree zero. For more details see \cite{[B2018]}.

We will now extend the Hall scalar product to fit with our notation in the following way, for $f$, $g$ in $\Lambda_Q$:
\begin{align*}\langle f\otimes s_\lambda, g\otimes s_\mu\rangle =\begin{cases}fg & \text{ if } \lambda=\mu 
\\ 0																	&\text{ if not} \end{cases}
\end{align*}

We will sometimes write $s_\mu$ instead of $1\otimes s_\mu$. Looking at our previous example we can easily see that:
\begin{align*}\langle\mathcal{E}_4, s_{1111}\rangle=\newcommand{\Bold}[1]{\mathbf{\otimes1}} s_{1 11}+ s_{3 1}+s_{41}+s_{6}
\end{align*}

Finally,  we  also need to recall that the dual Pieri rule describes the multiplication of a Schur function by $e_k$. The adjoint of the dual Pieri rule for the Hall scalar product, denoted $e_k^\perp$, is defined on the Schur basis and extended linearly. More precisely, $e_k^\perp s_\lambda$ is the sum  $\sum_\mu s_\mu$ over all partitions $\mu$ obtained by deleting $k$ boxes each lying in a different row (see Figure \ref{Fig : pieri}).

\begin{figure}[h!]
\centering
\begin{tikzpicture}[scale=.3]
\draw (0,0)--(0,4)--(1,4)--(1,2)--(3,2)--(3,1)--(4,1)--(4,0)--(0,0);
\draw (0,1)--(4,1);
\draw (0,2)--(3,2);
\draw (0,3)--(1,3);
\draw (1,0)--(1,4);
\draw (2,0)--(2,2);
\draw (3,0)--(3,1);
\node(e) at (-1,2){$e_2^\perp$};
\end{tikzpicture}
\begin{tikzpicture}[scale=.3]
\draw (0,0)--(0,4)--(1,4)--(1,2)--(2,2)--(2,1)--(3,1)--(3,0)--(0,0);
\draw (0,1)--(3,1);
\draw (0,2)--(2,2);
\draw (0,3)--(1,3);
\draw (1,0)--(1,4);
\draw (2,0)--(2,2);
\draw (3,0)--(3,1);
\node(e) at (-1,2){$=$};
\end{tikzpicture}
\begin{tikzpicture}[scale=.3]
\draw (0,0)--(0,3)--(1,3)--(1,2)--(3,2)--(3,0)--(0,0);
\draw (0,1)--(3,1);
\draw (0,2)--(3,2);
\draw (1,0)--(1,3);
\draw (2,0)--(2,2);
\draw (3,0)--(3,1);
\node(e) at (-1,2){$+$};
\end{tikzpicture}
\begin{tikzpicture}[scale=.3]
\draw (0,0)--(0,3)--(1,3)--(1,2)--(2,2)--(2,1)--(4,1)--(4,0)--(0,0);
\draw (0,1)--(4,1);
\draw (0,2)--(2,2);
\draw (1,0)--(1,3);
\draw (2,0)--(2,2);
\draw (3,0)--(3,1);
\node(e) at (-1,2){$+$};
\end{tikzpicture}
\begin{tikzpicture}[scale=.3]
\draw (0,0)--(0,2)--(3,2)--(3,1)--(4,1)--(4,0)--(0,0);
\draw (0,1)--(4,1);
\draw (0,2)--(3,2);
\draw (1,0)--(1,2);
\draw (2,0)--(2,2);
\draw (3,0)--(3,1);
\node(e) at (-1,2){$+$};
\end{tikzpicture}
\caption{$e_2^\perp s_{4311}=s_{3211}+s_{331}+s_{421}+s_{43}$}
\label{Fig : pieri}

\end{figure}

  
 \section{Lifting to multivariate formulas}\label{Sec : lift}
   
The following results gives us a way to lift to an alternating sum. We will show later  how to obtain positive sums from these.

  \begin{cl}Let $G\in \Lambda_Q\otimes \Lambda$, $G=\sum_\mu \mathcal{B}_{\mu}\otimes s_\mu$, where $\mathcal{B}_{\mu}=\sum_\lambda b_{\lambda,\mu}s_\lambda(Q)$, $b_{\lambda,\mu}\in\mathbb{N}$. If  $\lambda$ is a partition of shape $(a,b,1^k)$, with $k$ and $a$ arbitrary, and $b_{\lambda,\mu}\not=0$ then   the coefficient of $s_{(a,b-1)}(Q)$ in $e^{\perp}_{k+1}(\mathcal{B}_\mu)$ and the coefficient of $s_{(a-1,b-1)}(Q)$ in $e^{\perp}_{k+2}(\mathcal{B}_\mu)$ are not zero.
  
 \end{cl}
 
  \begin{proof}First notice that using the Pieri rule we have $e^{\perp}_{k+1}s_\lambda=s_{(a-1,b-1,1)}+s_{(a-1,b)}+s_{(a,b-1)}$ and $e^{\perp}_{k+2}s_\lambda=s_{(a-1,b-1)}$. Since the Pieri rule is linear and $c_{\lambda,\mu}\not=0$, we must have non-zero coefficients as claimed.  \end{proof}

 For the following lemma we first consider the application $\psi:\Lambda\rightarrow \mathbb{Q}[q,t]$ which is defined on the Schur basis by $\psi(s_\lambda)=q^{\lambda_1}t^{\ell(\lambda)-1}$ and extended linearly.  

The following lemma allows us to lift $\mathcal{E}^{\langle 2\rangle}_{m,n}$ to $\mathcal{E}_{m,n}$. Note that we can not obtain all the coefficients of $\mathcal{E}_{m,n}$ this way, but we know which are left out. If we consider the Schur decomposition of $\langle\mathcal{E}_{m,n},e_n\rangle$, the following lemma shows  how to obtain all the coefficients of the Schur functions indexed by partitions of shape  $(a,b+1,1^j)$, with $a$ and $j$ arbitrary and $b$ fixed.

   \begin{lem}\label{Lem : altern}Let $G \in \Lambda_Q$ be  a symmetric function. If $f_i(q)=\psi((e_i^\perp G)|_{1\textrm{part}})$, then:
   \begin{equation*} \psi(G|_{\hooks})(q,t)=\sum_{j\geq 0}\sum_{k=0}^j (-1)^k f_{j-k}(q)q^{-k}t^j
   \end{equation*}
 \end{lem}
 
  \begin{proof} If $\lambda$ is such that $s_\lambda$ has a non-zero coefficient in $G|_{\hooks}$ there exists a $k$ such that $\lambda=(a,1^k)$. Therefore: 
 \[e_i^\perp s_\lambda=s_{a,1^{k-i}}+s_{a-1,1^{k-i+1}}.\]
When $\mu=(b,1^l)$ is not a partition we set $s_\mu=0$. If one of these 3 conditions apply $b < 0$, $l<0$ or $b=0$ and $l\geq 1$ then $\mu$ is not a partition. 
The restriction to one part (or equivalently $1$ variables) keeps only the term of $e_i^\perp s_\lambda$ for which  $k-i+1=0$ and $a>1$ or $k-i=0$. Moreover, if $q^a$ as a non-zero coefficient in $f_i$ then $s_{(a)}$ as a non-zero coefficient in  $(e_i^\perp G)|_{1\textrm{part}}$, by definition of $f_i(q)$. So by the previous claim $q^a$ is associated to a $\lambda$ that contributed $q^{a+1}$ in $f_{i-1}$ or $q^{a-1}$ in $f_{i+1}$. Using the inclusion exclusion principal, the sum $\sum_{k=0}^j (-1)^k f_{j-k}(q)q^{-k}$ gives the monomials in $f_j$ that are associated to a terms $s_\lambda$ in $G|_{\hooks}$ which contributes $q^{a-1}$ in $f_{j+1}$. Consequently $\lambda$ must be of length $j+1$ in $G|_{\hooks}$ and  $\sum_{k=0}^j (-1)^k f_{j-k}(q)q^{-k}$ is the coefficient of $t^j$ in $\psi(G|_{\hooks})(q,t)$ has claimed.
 \end{proof}
 
 This last lemma can be generalized as follows.
 
  \begin{lem}\label{Lem : inclu-exclu}Let $b$ be a constant in $\mathbb{N}^*$ and $G \in \Lambda_Q$ be  a symmetric function in the variables $Q=\{q_1,q_2,\ldots\}$. Let  $\mathcal{V}_b$ be the set of partitions of shape $(a,b,1^k)$ with $k$ and $a$ arbitrary. If $f_0(q)=\psi(G|_{\mathcal{V}_{b}}^{\langle 2\rangle})t^{-1}$ and
   \begin{equation*}
   f_i(q)=\psi((e_i^\perp G|_{\mathcal{V}_{b}})^{\langle 2\rangle}|_{\mathcal{V}_{b}})t^{-1},
   \end{equation*}  then:
   \begin{equation*} \psi(G|_{\mathcal{V}_{b+1}})(q,t)=\sum_{j\geq 1}\sum_{k=0}^j (-1)^k f_{j-k}(q)q^{-k}t^j
   \end{equation*}
 \end{lem}
 
  \begin{proof} 
 
The difference is mainly that if $\lambda$ is such that $s_\lambda$ has a non-zero coefficient in $G|_{\mathcal{V}_{b+1}}$ there exists a $k$ such that $\lambda=a,b+1,1^k$. Therefore:  
 \[e_i^\perp s_\lambda=s_{a,b+1,1^{k-i}}+s_{a,b,1^{k-i+1}}+s_{a-1,b+1,1^{k-i+1}}+s_{a-1,b,1^{k-i+2}}\]
Since we have noticed before that the restriction to $2$ variables is equivalent to the restrictions to the sum over Schur functions indexed by partitions of length at most $2$. This means that we only keep the terms of $e_i^\perp s_\lambda$ such that $k+1= i$ and $a\geq b$ or $k+2=i$ and $a>b$. So:
\begin{equation*}(e_i^\perp s_\lambda)^{\langle 2\rangle}|_{\mathcal{V}_{b}}= \begin{cases}s_{a,b,1^{k-i+1}} &\text{ if $k+1=i$ } 
\\	s_{a-1,b,1^{k-i+2}}&\text{ if $k+2=i$ }
\\      0 &\text{ otherwise.}
\end{cases}
\end{equation*}
The remainder of the proof is similar to the previous lemma.
 \end{proof}

Note that $ \psi(G|_{\hooks})(q,t)$ is the sum of the restriction to $\mathcal{V}_1$ and the restriction to one parts ($|_{\mathcal{V}_0}$). This is the reason why, in Section \ref{Sec : new object}, we will use the convention that $T_{n,s}=\{\epsilon\}$ if $s\geq n-2$. The path $\epsilon$ relates to the restriction to one part.
 
 We should also notice that when $b\geq2$, no formula written in the Schur functions in the variables $q$ and $t$ is known for $\langle \nabla(e_n), s_{k,1^{n-k}}\rangle|_{\mathcal{V}_b}$ at this moment. Given this formula, the lemma gives a way to find the formulas for $\langle \mathcal{E}_{n,n}, s_{1^n}\rangle|_{\mathcal{V}_{b+1}}$.  
 
The restriction of $\mathcal{E}_{rn,n}$ to a set of two variables is predicted to be $\nabla^r(e_n)$ in \cite{[B2018]}. The equality $\langle\mathcal{E}_{n,n}, s_{j+1,1^{n-j}}\rangle=e^\perp_j \langle\mathcal{E}_{n,n}, s_{1^{n}}\rangle$ is conjectured to also hold for $\langle\mathcal{E}_{rn,n}, s_{j+1,1^{n-j}}\rangle=e^\perp_j \langle\mathcal{E}_{rn,n}, s_{1^{n}}\rangle$. Using the combinatorics of $m$-Schr\"oder paths, the author found the following $q$-analogue in \cite{[Wal2019b]}:
 \begin{equation}\label{Eq : 1part r}f_j^{(r)}(q):=\psi( \langle\nabla^r(e_n), s_{j+1,1^{n-j-1}}\rangle|_{\textrm{1 part}})=q^{r\binom{n}{2}-\binom{j+1}{2}}\begin{bmatrix}n-1\\ j \end{bmatrix}_{q^{-1}}.
 \end{equation}
That result and the last lemma will be used to find the hook components for  $ \langle\mathcal{E}_{n,n}, s_{1^{n}}\rangle$ (and conjectured hook components for  $ \langle\mathcal{E}_{rn,n}, s_{1^{n}}\rangle$). But first we need to introduce a combinatorial object that will help us transform the alternating sum into a positive sum.
 
  
 \section{New combinatorial object}\label{Sec : new object}
 
 Our formula for $\langle \mathcal{E}_{rn,n}, s_{\mu}\rangle|_{\hooks}$ will be formulated in terms of a new objects that we denote by  $T_{n,s}$. We will often write  $T_n$ for $T_{n,0}$. These objects help to transform an alternating sum into a positive sum and has a $q,z$-analogue, $T_n(q,z)$, realizing the $q$-Pochhammer symbol (or $q$-rising factorial) $(-qz;q)_{n-2}=\prod_{i=1}^{n-2} (1+zq^i)$. Notice that the substitution $T_n(q,-zq^{-1})$ bring us back to the usual way of seeing the $q$-Pochhammer symbol $(z;q)_{n-2}=\prod_{i=0}^{n-3} (1-zq^i)$.

Let $T_{n,s}$ denote the set of north-east paths in an $n-2$ staircase shaped grid lying in $\mathbb{N}^2$, starting at $(0,s)$ and ending at a point in the set $\{ (x,y)~|~x+y=n-2 , x\geq 0\text{ and } y\geq s\}$. For an example see Figure \ref{Fig : tns}. The relevant paths can be represented as words of length $n-s-2$ in the alphabet $\{N,E\}$. For reasons stated earlier when $s>n-1$ we set $s=n-2$. In that case $T_{n,s}=\{\epsilon\}$, where $\epsilon$ is the empty word. Note that $\A(\epsilon)=\binom{n-1}{2}$ and $\h(\epsilon)=n-2$. Notice that for $n<2$, $T_{n,s}=\emptyset$.

The \begin{it}area\end{it} of a path, denoted $\A$, is the number of boxes south-east of the path (see Figure \ref{Fig : aire}). The \begin{it}height\end{it} of a path is the $y$ coordinate of its end point (see Figure \ref{Fig : ht}).

\begin{figure}[h!]
\centering
\begin{minipage}{4.2cm}
\begin{tikzpicture}[scale=.4]
\draw [thin, gray](0,0)--(0,5)--(1,5)--(1,4)--(2,4)--(2,3)--(3,3)--(3,2)--(4,2)--(4,1)--(5,1)--(5,0)--(0,0);
\draw [thin, gray](1,0)--(1,4);
\draw [thin, gray](2,0)--(2,3);
\draw [thin, gray](3,0)--(3,2);
\draw [thin, gray](4,0)--(4,1);
\draw [thin, gray](0,1)--(4,1);
\draw [thin, gray](0,2)--(3,2);
\draw [thin, gray](0,3)--(2,3);
\draw [thin, gray](0,4)--(1,4);
\node(d) at (0,2){\textcolor{blue}{$\bullet$}};
\node(deb) at (-2.5,2){\textcolor{blue}{start}};
\draw[->, blue] (deb)--(d);
\node(f3) at (3,2){\textcolor{red}{$\bullet$}};
\node(f4) at (2,3){\textcolor{red}{$\bullet$}};
\node(f5) at (1,4){\textcolor{red}{$\bullet$}};
\node(f6) at (0,5){\textcolor{red}{$\bullet$}};
\node(fin) at (5.5,5.5){\textcolor{red}{finish}};
\draw[->, red] (fin)--(f3);
\draw[->, red] (fin)--(f4);
\draw[->, red] (fin)--(f5);
\draw[->, red] (fin)--(f6);
\draw[<->] (-0.5,0)--(-.5,1.8);
\node(p) at (-1,1){$s$};
\draw[<->] (0,-.5)--(5,-.5);
\node(n) at (2.5,-1){$n-2$};
\end{tikzpicture}
\caption{$T_{7,2}$}
\label{Fig : tns}
\end{minipage}
\begin{minipage}{6.3cm}
\centering
\begin{tikzpicture}[scale=.5]
\filldraw[pink] (0,2)--(0,3)--(1,3)--(1,4)--(2,4)--(2,3)--(3,3)--(3,2)--(4,2)--(4,1)--(5,1)--(5,0)--(0,0);
\draw[thin, gray] (0,0)--(0,5)--(1,5)--(1,4)--(2,4)--(2,3)--(3,3)--(3,2)--(4,2)--(4,1)--(5,1)--(5,0)--(0,0);
\draw[thin, gray] (1,0)--(1,4);
\draw[thin, gray] (2,0)--(2,3);
\draw[thin, gray] (3,0)--(3,2);
\draw[thin, gray] (4,0)--(4,1);
\draw[thin, gray] (0,1)--(4,1);
\draw[thin, gray] (0,2)--(3,2);
\draw[thin, gray] (0,3)--(2,3);
\draw[thin, gray] (0,4)--(1,4);
\draw[blue, thick] (0,2)--(0,3)--(1,3)--(1,4);
\node(d) at (0,2){\textcolor{blue}{$\bullet$}};
\node(f3) at (3,2){\textcolor{red}{$\bullet$}};
\node(f4) at (2,3){\textcolor{red}{$\bullet$}};
\node(f5) at (1,4){\textcolor{red}{$\bullet$}};
\node(f6) at (0,5){\textcolor{red}{$\bullet$}};
\node(arranger le bas) at (2.5,-.6){};
\end{tikzpicture}
\caption{ $\A(NEN)=13$}
\label{Fig : aire}
\end{minipage}
\begin{minipage}{5.75cm}
\centering
\begin{tikzpicture}[scale=.5]
\node(fin) at (5.5,6){};
\draw[thin, gray] (0,0)--(0,5)--(1,5)--(1,4)--(2,4)--(2,3)--(3,3)--(3,2)--(4,2)--(4,1)--(5,1)--(5,0)--(0,0);
\draw[thin, gray] (1,0)--(1,4);
\draw[thin, gray] (2,0)--(2,3);
\draw[thin, gray] (3,0)--(3,2);
\draw[thin, gray] (4,0)--(4,1);
\draw[thin, gray] (0,1)--(4,1);
\draw[thin, gray] (0,2)--(3,2);
\draw[thin, gray] (0,3)--(2,3);
\draw[thin, gray] (0,4)--(1,4);
\draw[blue, thick] (0,2)--(0,3)--(1,3)--(1,4);
\node(d) at (0,2){\textcolor{blue}{$\bullet$}};
\node(deb) at (-1.5,2){};
\node(f3) at (3,2){\textcolor{red}{$\bullet$}};
\node(f4) at (2,3){\textcolor{red}{$\bullet$}};
\node(f5) at (1,4){\textcolor{red}{$\bullet$}};
\node(f6) at (0,5){\textcolor{red}{$\bullet$}};
\draw[dotted] (1,4)--(6,4);
\draw[<->] (6,0)--(6,4);
\node(h) at (7.5,2){$ht(\gamma)$};
\node(arranger le bas) at (2.5,-.6){};
\end{tikzpicture}
\caption{ $\h(NEN)=4$}
\label{Fig : ht}
\end{minipage}
\end{figure}

Many known objects are in bijection (compositions, words, self conjugate partitions, partitions with distinct parts, partitions with distinct odd parts, $321$ and $231$ avoiding permutations) with $T_{n,s}$, a curious reader could see \cite{[Wal2019c]}.  Note that the image of $T_{n,s}$  through These bijections preserve many statistics on the associated objects and in some cases they can
be used to refine the sets. It also contains another proof of  the Lagrange convolution using $T_{n,s}$.

This new object has the following generating function.

\begin{prop}\label{Prop : T_n}Let $T_{n,s}(q,z)=\sum_{\gamma\in T_{n,s}} q^{\A(\gamma)}z^{\h(\gamma)}$. Then for $r=n-s-2$ we have:
\begin{equation}\label{Eq : Tr0->Tns}
  T_{n,s}(q,z)=\sum_{j=0}^{r} q^{\binom{s+j+1}{2}+s(r-j)}\begin{bmatrix}r \\ j\end{bmatrix}_q z^{j+s}=T_{r+2,0}(q,z)z^{s}q^{rs+\binom{s+1}{2}}
 \end{equation}
In particular if $s=0$ we have:
\begin{equation}\label{Eq : Tn(q,z)}
 T_{n,0}(q,z)=\sum_{j=0}^{n-2} q^{\binom{j+1}{2}}\begin{bmatrix}n-2 \\ j\end{bmatrix}_q z^{j}=\sum_{k=0}^{n-2} q^{\binom{n-k-1}{2}}\begin{bmatrix}n-2 \\ n-2-k\end{bmatrix}_q z^{n-2-k}=(-qz;q)_{n-2}
 \end{equation}
\end{prop}
Note that by our choice of convention when $s\geq n-2$ we have $T_{n,s}(q,z)=q^{\binom{n-1}{2}}z^{n-2}$.
\begin{proof}Starting with $T_{n,0}(q,z)$, we only need to prove the first equality since the second equality is the change of variables $k=n-2-j$ and the last equality is the well know $q$-binomial theorem. The paths ending at height $j$ are the paths that fit in a $j\times (n-j-2)$ grid. It is known that the $q$-analogue of these paths with its respective area statistic are the gaussian $q$-binomial, $\begin{bmatrix}n-2 \\ j\end{bmatrix}_q$. This leaves us with a staircase of height $j$ which contains ${\binom{j+1}{2}}$ boxes. In consequence the  coefficient of $z^j$ is $q^{\binom{j+1}{2}}\begin{bmatrix}n-2 \\ j\end{bmatrix}_q$. 

For  $T_{n,s}$ we need only to notice that there is a natural bijection with the paths of $T_{n-s,0}$. The only difference is the statistics. The height statistic is exactly $s$ greater in $T_{n,s}$ and the area statistic is exactly $(n-2-s)s+\binom{s+1}{2}$ in $T_{n,s}$. Hence, Equation \eqref{Eq : Tr0->Tns}.
\end{proof}

We will see that the reason $T_{n,s}(q,z)$ is useful to transform the alternating sums, induced in  Lemma \ref{Lem : altern}, into a positive sum is related to the exponents of $q$ in Equation \eqref{Eq : Tr0->Tns}. The exponents have the following property.

\begin{cl}Let $c$ be an integer and $g_j:\mathbb{N}\rightarrow \mathbb{Z}$ be maps such that $g_j(k)-g_j(k-1)=k+j+c$ for all $k\geq 1$. Then $g_j(k)=k(j+c)+\binom{k+1}{2}+g_j(0)$ for all $j\geq 0$ and all $k\geq 1$.
\end{cl}
\begin{proof}By definition of the maps we know that $g_j(1)=j+1+c+g_j(0)$. Since  $g_j(k)-g_j(k-1)=k+j+c$ by induction we have $g_j(k)=k+j+c+(k-1)(j+c)+\binom{k}{2}+g_j(0)$ which completes the proof.
\end{proof}

The following result shows how this object is used to eliminate an alternating sum and make the relevant formula positive.

\begin{prop}\label{Prop : T_n=sum}Let $g_j:\mathbb{N}\rightarrow \mathbb{Z}$ be such that $g_j(k)-g_j(k-1)=k+j$ for all $k\geq 1$. Then:
\begin{equation}\label{Eq : alt->pos}\sum_{k=0}^{n-j-1}(-1)^k\begin{bmatrix}n-1\\j+k\end{bmatrix}_qq^{-k+g_j(k)}=\begin{bmatrix}n-2\\j-1\end{bmatrix}_qq^{g_j(0)}\end{equation}
 and: 
 \begin{equation}\label{Eq : nulle} \sum_{k=0}^{n-1}(-1)^k\begin{bmatrix}n-1\\k\end{bmatrix}_qq^{-k+g_0(k)}=0.
\end{equation}

Moreover, if $g_{j}(0)-\binom{j}{2}=g_{i}(0)-\binom{i}{2}$ for all $i,j$, then for any $j$ we have:
\begin{equation}\label{Eq : somme alt->pos}\sum_{j=1}^{n-1}\sum_{k=0}^{n-j-1}(-1)^k\begin{bmatrix}n-1\\j+k\end{bmatrix}_qq^{-k+g_j(k)}z^{j-1}=T_{n,0}(q,z)q^{g_{j}(0)-\binom{j}{2}}, 
\end{equation}
and if $g_{j}(0)-\binom{j+1}{2}=g_{i}(0)-\binom{i+1}{2}$ for all $i,j$, then for any $j$ we have:
\begin{equation}\label{Eq : somme alt->pos air+ht}\sum_{j=1}^{n-1}\sum_{k=0}^{n-j-1}(-1)^k\begin{bmatrix}n-1\\j+k\end{bmatrix}_qq^{-k+g_j(k)}z^{j-1}=T_{n,0}(q,qz)q^{g_{j}(0)-\binom{j+1}{2}+1}.
\end{equation}
\end{prop}

Before we give the proof we will give a combinatorial intuition based on the case $g_{j}(k)=\binom{j+k+1}{2}$. Notice that by the previous proposition we only need to prove the second part. For some fixed $k$ and $j$ we can represent the term $\left[{n-1\atop j+k}\right]_qq^{-k+g_j(k)}$ by the set of paths in $\mathcal{C}^{n-1}_{j+k}$ to which we add $g_j(k)$ boxes and subtract $k$ boxes (see Figure \ref{Fig : g_j(k)}).
There is a bijection between  paths ending with a north step in $\mathcal{C}^{n-1}_{j+k+1}$ (blue in Figure \ref{Fig : boxes}) and paths ending with an east step in $\mathcal{C}^{n-1}_{j+k}$ (red in  Figure \ref{Fig : boxes}). We only need to change the last step. This is an involution and they both account for the same number of boxes (see Figure \ref{Fig : boxes} as an example). Since the terms have coefficient $(-1)^k$ they cancel out pairwise in the sum. So the only steps left are the ones when $k=0$ and the path in $\mathcal{C}^{n-1}_{j}$ ends with a north step. Eliminating the last north step doesn't affect the area because there are no east steps afterwards. Therefore, we can consider the paths in $\mathcal{C}^{n-2}_{j-1}$ with the same statistic, which is what we needed. Note that for $j=0$ there is no path in $\mathcal{C}_0^{n-1}$ that ends with a north step.

\begin{figure}[h!]
\begin{minipage}{8.25cm}
\centering
\begin{tikzpicture}[scale=.5]
\filldraw[gray!50] (0,4)--(0,6)--(1,6)--(1,5)--(2,5)--(2,4)--(0,4);
\filldraw[pink] (2,0)--(2,4)--(3,4)--(3,3)--(4,3)--(4,2)--(5,2)--(5,1)--(6,1)--(6,0)--(2,0);
\draw[thin, gray] (0,0)--(0,6)--(1,6)--(1,5)--(2,5)--(2,4)--(3,4)--(3,3)--(4,3)--(4,2)--(5,2)--(5,1)--(6,1)--(6,0)--(0,0);
\draw[thin, gray] (1,0)--(1,5);
\draw[thin, gray] (2,0)--(2,4);
\draw[thin, gray] (3,0)--(3,3);
\draw[thin, gray] (4,0)--(4,2);
\draw[thin, gray] (5,0)--(5,1);
\draw[thin, gray] (0,1)--(5,1);
\draw[thin, gray] (0,2)--(4,2);
\draw[thin, gray] (0,3)--(3,3);
\draw[thin, gray] (0,4)--(2,4);
\draw[thin, gray] (0,5)--(1,5);
\draw[red] (0,0)--(2,0)--(2,4)--(0,4)--(0,0);
\node(d) at (0,0){\textcolor{black}{$\bullet$}};
\node(f4) at (2,4){\textcolor{red}{$\bullet$}};
\node(k) at (-1,3){$k$};
\draw[<->] (-.5,2)--(-.5,4);
\node(p) at (-1,1){$j$};
\draw[<->] (-.5,0)--(-.5,2);
\draw[<->] (0,-.5)--(6,-.5);
\draw (-1,2)--(6,2);
\node(n) at (3,-1){$n-1$};
\node(x1) at (3.5,2.5){X};
\node(x2) at (2.5,3.5){X};
\node(g) at (5,4){\textcolor{red!55}{$g_{j}(k)$}};
\draw[->] (g)--(3.75,1.5);
\end{tikzpicture}

\caption{Representation of the term $\left[{n-1 \atop j+k}\right]_qq^{-k+g_j(k)}$ in the Equation \eqref{Eq : alt->pos}}
\label{Fig : g_j(k)}
\end{minipage}
\begin{minipage}{8.25cm}
\centering
\begin{tikzpicture}[scale=.5]
\filldraw[blue!20] (1,0)--(2,0)--(2,5)--(1,5)--(1,0);
\filldraw[pink] (2,0)--(2,4)--(3,4)--(3,3)--(4,3)--(4,2)--(5,2)--(5,1)--(6,1)--(6,0)--(2,0);
\draw[thin, gray] (0,0)--(0,6)--(1,6)--(1,5)--(2,5)--(2,4)--(3,4)--(3,3)--(4,3)--(4,2)--(5,2)--(5,1)--(6,1)--(6,0)--(0,0);
\draw[thin, gray] (1,0)--(1,5);
\draw[thin, gray] (2,0)--(2,4);
\draw[thin, gray] (3,0)--(3,3);
\draw[thin, gray] (4,0)--(4,2);
\draw[thin, gray] (5,0)--(5,1);
\draw[thin, gray] (0,1)--(5,1);
\draw[thin, gray] (0,2)--(4,2);
\draw[thin, gray] (0,3)--(3,3);
\draw[thin, gray] (0,4)--(2,4);
\draw[thin, gray] (0,5)--(1,5);
\node(d) at (0,0){\textcolor{black}{$\bullet$}};
\node(f4) at (2,4){\textcolor{red}{$\bullet$}};
\draw[very thick,red] (1,4)--(2,4);
\draw[very thick,blue] (1,4)--(1,5);
\node(f5) at (1,5){\textcolor{blue}{$\bullet$}};
\node(k) at (-1,3){$k$};
\draw[<->] (-.5,2)--(-.5,4);
\node(p) at (-1,1){$j$};
\draw[<->] (-.5,0)--(-.5,2);
\draw (-1,2)--(6,2);
\node(x1) at (1.5,4.5){\textcolor{gray}{X}};
\node(x2) at (2.5,3.5){X};
\node(x3) at (3.5,2.5){X};
\end{tikzpicture}

\caption{Comparing path ending with a North step in $\mathcal{C}^{n-1}_{j+k+1}$ and with an east step in $\mathcal{C}^{n-1}_{j+k}$}
\label{Fig : boxes}
\end{minipage}
\end{figure}

\begin{proof}Since $\binom{j+k+1}{2}-\binom{j+1}{2}=\binom{k+1}{2}+kj$,  by the previous claim we can rewrite the left hand of Equation \eqref{Eq : alt->pos}:
\begin{equation}\label{Eq : alt->pos pre}q^{g_j(0)-\binom{j+1}{2}}\sum_{k=0}^{n-j-1}(-1)^k\begin{bmatrix}n-1\\j+k\end{bmatrix}_qq^{-k+\binom{j+k+1}{2}}
\end{equation}
Let us consider each path of $T_{n+1,0}$ that end over the line of equation $y=j$ and weight them differently:
\begin{equation*}
\hat T_{n+1,0}^{(j)}(q,z):=\underset{ \h(\gamma)\geq j}{\sum_{\gamma\in T_{n+1,0}}} (-qz)^{j-\h(\gamma)}q^{\A(\gamma)}z^{\h(\gamma)}
\end{equation*}
Using the same proof as in Proposition \ref{Prop : T_n} we get:
 \begin{equation*}
 \hat{T}_{n+1,0}^{(j)}(q,z)=\sum_{k=0}^{n-j-1}(-1)^{k} q^{\binom{j+k+1}{2}-k}\begin{bmatrix}n-1 \\ j+k\end{bmatrix}_q z^{j}.
 \end{equation*}
 Notice that Equation \eqref{Eq : alt->pos pre} is equal to $q^{g_j(0)-\binom{j+1}{2}} \hat{T}_{n+1,0}^{(j)}(q,z)z^{-j}$.

Because each paths cross the line $y=j$ we can consider all the paths that touch the line $y=j$ for the first time at $(n-1-j-l,j)$. These end to the north east of $(n-1-j-l,j)$ to any end point among the set $\{ (n-1-j-l+k,j+l-k)~|~0\leq k\leq l\}$. But the set of paths starting at $(n-1-j-l,j)$ end ending at  $\{ (n-1-j-l+k,j+l-k)~|~0\leq k\leq l\}$ are the paths of $T_{l+2,0}$ translated by $(n-1-j-l,j)$. Furthermore the paths of height $l-k$ in $T_{l+2,0}$ translated by $(n-1-j-l,j)$ end at $(n-1-j-l+k,j+l-k)$ ergo the related monomials are skewed by $(-qz)^{-l+k}$, in $\hat T_{n+1,0}^{(j)}(q,z)$. Therefore, if we consider the part of the path over the line $y=j$ we see that it is weighted by $\hat T_{l+2,0}^{(0)}(q,z)$. Take note that $\hat T_{l+2,0}^{(0)}(q,z)$ is a polynomial in $q$ and not in $q$ and $z$ since $l-k$ is the height of the path. 

\begin{figure}[h!]
\begin{minipage}{9cm}
\centering
\begin{tikzpicture}[scale=.5]
\filldraw[gray!50] (0,0)--(0,6)--(1,6)--(1,5)--(2,5)--(2,0)--(0,0);
\filldraw[pink] (2,1)--(2,-2)--(8,-2)--(8,-1)--(7,-1)--(7,0)--(6,0)--(6,1)--(2,1);
\draw[thin, gray] (0,-2)--(0,6)--(1,6)--(1,5)--(2,5)--(2,4)--(3,4)--(3,3)--(4,3)--(4,2)--(5,2)--(5,1)--(6,1)--(6,0)--(7,0)--(7,-1)--(8,-1)--(8,-2)--(0,-2);
\draw[thin, gray] (1,-2)--(1,5);
\draw[thin, gray] (2,-2)--(2,4);
\draw[thin, gray] (3,-2)--(3,3);
\draw[thin, gray] (4,-2)--(4,2);
\draw[thin, gray] (5,-2)--(5,1);
\draw[thin, gray] (6,-2)--(6,0);
\draw[thin, gray] (7,-2)--(7,-1);
\draw[thin, gray] (0,-1)--(7,-1);
\draw[thin, gray] (0,0)--(6,0);
\draw[thin, gray] (0,1)--(5,1);
\draw[thin, gray] (0,2)--(4,2);
\draw[thin, gray] (0,3)--(3,3);
\draw[thin, gray] (0,4)--(2,4);
\draw[thin, gray] (0,5)--(1,5);
\draw[thick, blue] (2,0)--(2,-2)--(0,-2)--(0,0)--(2,0);
\draw[thick, red] (2,1)--(2,4)--(3,4)--(3,3)--(4,3)--(4,2)--(5,2)--(5,1)--(2,1);
\draw (2,0)--(2,1);
\node(dr) at (0,-2){\textcolor{blue}{$\bullet$}};
\node(fr) at (2,0){\textcolor{blue}{$\bullet$}};
\node(db) at (2,1){\textcolor{red}{$\bullet$}};
\node(f1) at (2,4){\textcolor{red}{$\bullet$}};
\node(f2) at (3,3){\textcolor{red}{$\bullet$}};
\node(f3) at (4,2){\textcolor{red}{$\bullet$}};
\node(f4) at (5,1){\textcolor{red}{$\bullet$}};
\node(l) at (3.5,0){$l$};
\draw[<->] (2,0.5)--(5,0.5);
\node(T) at (-3.5,2){$(n-1-j-l,j)$};
\draw[->] (T)--(db);
\node(p) at (-1,-.5){$j$};
\draw[<->] (-.5,-2)--(-.5,1);
\draw[<->] (0,-2.5)--(8,-2.5);
\draw [thin] (-1,1)--(7,1);
\node(n) at (4,-3){$n-1$};
\node(T) at (5,4){\textcolor{red}{$T_{l+2,0}$}};
\draw[red!55,->] (T)--(3.75,1.5);
\node(bin) at (10,-.5){\textcolor{red!55}{$\binom{j+1}{2}+jl$}};
\draw[red!55,->] (bin)--(5.25,-.75);
\node(C) at (-3,-1){\textcolor{blue!55}{$\mathcal{C}_{j-1}^{n-2-l}$}};
\draw[blue!55,->] (C)--(.75,-1.25);
\end{tikzpicture}
\caption{Concatenation of a path of $T_{l+2,0}$ and a path of $\mathcal{C}_{j-1}^{n-2-l}$}
\label{Fig : hat T_n}
\end{minipage}\begin{minipage}{9cm}
\centering
\begin{tikzpicture}[scale=.5]
\filldraw[gray!50] (0,0)--(0,6)--(1,6)--(1,5)--(2,5)--(2,4)--(3,4)--(3,3)--(4,3)--(4,2)--(5,2)--(5,1)--(5,0)--(0,0);
\filldraw[pink] (5,-2)--(8,-2)--(8,-1)--(7,-1)--(7,0)--(6,0)--(6,1)--(5,1)--(5,-2);
\draw[thin, gray] (0,-2)--(0,6)--(1,6)--(1,5)--(2,5)--(2,4)--(3,4)--(3,3)--(4,3)--(4,2)--(5,2)--(5,1)--(6,1)--(6,0)--(7,0)--(7,-1)--(8,-1)--(8,-2)--(0,-2);
\draw[thin, gray] (1,-2)--(1,5);
\draw[thin, gray] (2,-2)--(2,4);
\draw[thin, gray] (3,-2)--(3,3);
\draw[thin, gray] (4,-2)--(4,2);
\draw[thin, gray] (5,-2)--(5,1);
\draw[thin, gray] (6,-2)--(6,0);
\draw[thin, gray] (7,-2)--(7,-1);
\draw[thin, gray] (0,-1)--(7,-1);
\draw[thin, gray] (0,0)--(6,0);
\draw[thin, gray] (0,1)--(5,1);
\draw[thin, gray] (0,2)--(4,2);
\draw[thin, gray] (0,3)--(3,3);
\draw[thin, gray] (0,4)--(2,4);
\draw[thin, gray] (0,5)--(1,5);
\draw[thick, blue] (5,0)--(5,-2)--(0,-2)--(0,0)--(5,0)--(5,1);
\node(dr) at (0,-2){\textcolor{blue}{$\bullet$}};
\node(fr) at (5,1){\textcolor{blue}{$\bullet$}};
\node(p) at (-1,-.5){$j$};
\draw[<->] (-.5,-2)--(-.5,1);
\draw[<->] (0,-2.5)--(8,-2.5);
\draw [thin] (-1,1)--(7,1);
\node(n) at (4,-3){$n-1$};
\node(bin) at (10,-.5){\textcolor{red!55}{$\binom{j+1}{2}$}};
\draw[red!55,->] (bin)--(6.4,-1.25);
\node(C) at (-3,-1){\textcolor{blue!55}{$\mathcal{C}_{j-1}^{n-2}$}};
\draw[blue!55,->] (C)--(.75,-1.25);
\end{tikzpicture}
\caption{The paths that are not cancelled out for $j$ (when $l=0$)}
\label{Fig : hat T_n l=0}
\end{minipage}
\end{figure}

Furthermore, the paths of $T_{n+1,0}$ crossing at $(n-1-l-j,j)$ are the concatenation of a path of $T_{l+2,0}$ and a path of $\mathcal{C}_{j-1}^{n-2-l}$, since the paths contain the north step that starts at  $(n-1-l-j,j-1)$ (see Figure \ref{Fig : hat T_n}). 
Hence, Equation \eqref{Eq : alt->pos pre} is equivalent to:
\begin{equation}\label{Eq : avec hat T} q^{g_j(0)-\binom{j+1}{2}}\sum_{l=0}^{n-j-1}\begin{bmatrix}n-2-l\\j-1\end{bmatrix}_q\hat T_{l+2,0}^{(0)}(q,z) q^{jl+\binom{j+1}{2}}
\end{equation}

Now by Proposition \ref{Prop : T_n} we have:
\begin{equation}\label{Eq : Tl+2} T_{l+2,0}(q,-zq^{-1})=\sum_{k=0}^{l} q^{\binom{l-k+1}{2}}\begin{bmatrix}l \\ l-k\end{bmatrix}_q z^{l-k}(-q)^{-l+k}
\end{equation}
 and by definition of our skewed sum we have:
 \begin{equation}\label{Eq : hatTl+2} \hat T_{l+2,0}^{(0)}(q,z)=\sum_{k=0}^{l} (-1)^{l-k}q^{\binom{l-k+1}{2}-l+k}\begin{bmatrix}l \\ l-k\end{bmatrix}_q z^{0}
 \end{equation}
  therefore comparing Equation \eqref{Eq : Tl+2} and Equation \eqref{Eq : hatTl+2} we get $\hat T_{l+2,0}^{(0)}(q,1)=T_{l+2,0}(q,-q^{-1}).$

By Proposition \ref{Prop : T_n} $T_{l+2,0}(q,-zq^{-1})=\prod_{i=0}^{l-1}(1-zq^i)$ thus $T_{l+2,0}(q,-q^{-1})=0$ if $l>0$ and $1$ if $l=0$. The equality $\hat T_{l+2,0}^{(0)}(q,1)=\hat T_{l+2,0}^{(0)}(q,z)$ means it is equivalent to state that $\hat T_{l+2,0}^{(0)}(q,z)=0$ unless $l=0$. This proves Equation \eqref{Eq : nulle}.

Replacing this in Equation \eqref{Eq : avec hat T} gives us Equation \eqref{Eq : alt->pos} (see Figure \ref{Fig : hat T_n l=0}). Equation \eqref{Eq : somme alt->pos} (respectively Equation \eqref{Eq : somme alt->pos air+ht})   follows from $g_j(0)-\binom{j}{2}=g_i(0)-\binom{i}{2}$ (respectively $g_j(0)-\binom{j+1}{2}=g_i(0)-\binom{i+1}{2}$) for all $i,j$ and Proposition \ref{Prop : T_n}.
\end{proof}

Notice that Equation \eqref{Eq : nulle} means that we can change the range of Proposition \ref{Prop : T_n=sum} for $0\leq j \leq n-1$ and obtain the same result.

We will see in the next section how to use this to get the formula for the hook components of $\langle \mathcal{E}_{rn,n}, e_n\rangle$ and of $\langle \mathcal{E}_{n,n}, s_{\mu}\rangle$. 

  
 \section{Schur positive explicit combinatorial Formula}\label{Sec : main}

We can now prove Theorem \ref{Thm : main} piece by piece.

Figure \ref{Fig : ex E_4} gives an example for $r=1$ and $n=4$. One might want to take note that for $n\leq4$ there are only hooks thus $\langle \mathcal{E}_{4,4}, e_4\rangle|_{hooks}=\langle \mathcal{E}_{4,4}, e_4\rangle$. This is why we can state in the introduction that  the equation of theorem 3.2.5 in \cite{[BCP2018]} is a specialization of Equation \eqref{Eq : main} using a different basis.

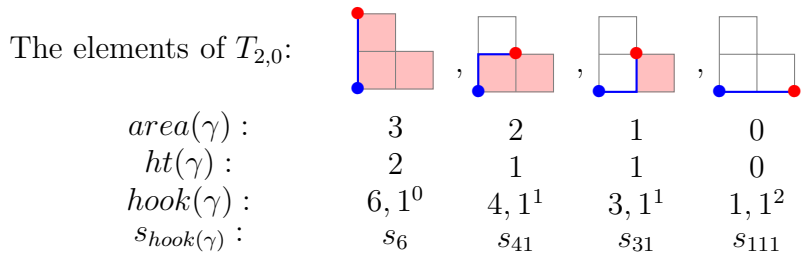
\begin{figure}[h!]
\centering
\begin{tikzpicture}[scale=.5]
\node(w) at (-5.5,1){The elements of $T_{2,0}$: };
\node(a) at (-4.5,-1){$area(\gamma):$ };
\node(h) at (-4.5,-2){$ht(\gamma): $};
\node(e) at (-4.5,-3){$hook(\gamma):$ };
\node(e) at (-4.5,-4){$s_{hook(\gamma)}:$ };
\filldraw[pink] (0,0)--(0,2)--(1,2)--(1,1)--(2,1)--(2,0)--(0,0);
\draw[thin,gray] (0,0)--(0,2)--(1,2)--(1,1)--(2,1)--(2,0)--(0,0);
\draw[thin,gray] (1,0)--(1,1)--(0,1);
\draw[blue, thick]  (0,0)--(0,2);
\node(f6) at (0,2){\textcolor{red}{$\bullet$}};
\node(d) at (0,0){\textcolor{blue}{$\bullet$}};
\node(a1) at (1,-1){$3$};
\node(h1) at (1,-2){$2$};
\node(H1) at (1,-3){$6,1^{0}$};
\node(s) at (1,-4){$s_6$};
\end{tikzpicture}
\begin{tikzpicture}[scale=.5]
\node(s) at (-.5,.5){,};
\filldraw[pink] (0,0)--(0,1)--(2,1)--(2,0)--(0,0);
\draw[thin,gray] (0,0)--(0,2)--(1,2)--(1,1)--(2,1)--(2,0)--(0,0);
\draw[thin,gray] (1,0)--(1,1)--(0,1);
\draw[blue, thick]  (0,0)--(0,1)--(1,1);
\node(f6) at (1,1){\textcolor{red}{$\bullet$}};
\node(d) at (0,0){\textcolor{blue}{$\bullet$}};
\node(a1) at (1,-1){$2$};
\node(h1) at (1,-2){$1$};
\node(H1) at (1,-3){$4,1^{1}$};
\node(s) at (1,-4){$s_{41}$};
\end{tikzpicture}
\begin{tikzpicture}[scale=.5]
\node(w) at (-.5,.5){,};
\filldraw[pink] (1,0)--(1,1)--(2,1)--(2,0)--(1,0);
\draw[thin,gray] (0,0)--(0,2)--(1,2)--(1,1)--(2,1)--(2,0)--(0,0);
\draw[thin,gray] (1,0)--(1,1)--(0,1);
\draw[blue, thick]  (0,0)--(1,0)--(1,1);
\node(f6) at (1,1){\textcolor{red}{$\bullet$}};
\node(d) at (0,0){\textcolor{blue}{$\bullet$}};
\node(a1) at (1,-1){$1$};
\node(h1) at (1,-2){$1$};
\node(H1) at (1,-3){$3,1^{1}$};
\node(s) at (1,-4){$s_{31}$};
\end{tikzpicture}
\begin{tikzpicture}[scale=.5]
\node(w) at (-.5,.5){,};
\draw[thin,gray] (0,0)--(0,2)--(1,2)--(1,1)--(2,1)--(2,0)--(0,0);
\draw[thin,gray] (1,0)--(1,1)--(0,1);
\draw[blue, thick]  (0,0)--(2,0);
\node(f6) at (2,0){\textcolor{red}{$\bullet$}};
\node(d) at (0,0){\textcolor{blue}{$\bullet$}};
\node(a1) at (1,-1){$0$};
\node(h1) at (1,-2){$0$};
\node(H1) at (1,-3){$1,1^{2}$};
\node(s) at (1,-4){$s_{111}$};
\end{tikzpicture}
  \begin{align*}\langle \mathcal{E}_{4,4}, e_4\rangle|_{hooks}&= s_{6}+s_{41}+s_{31}+s_{111}
 \end{align*}
\caption{Example of $\langle \mathcal{E}_{4,4}, e_4\rangle|_{hooks}$}
\label{Fig : ex E_4}
\end{figure}

For those who are used to seeing $\langle \mathcal{E}^{\langle 2\rangle}_{n,n}, s_{\mu}\rangle=\langle \nabla(e_n), s_{\mu}\rangle$ in terms of Dyck paths and Schr\"oder paths, note that each path in $T_{n,k}$ is associated to a subset of Schr\"oder paths.
The next proposition is related to the restriction of Equation \eqref{Eq : main} to $\mathcal{E}^{\langle 2\rangle}_{n,n}$. Haglund proved in  \cite{[H2004]} that $\langle \nabla(e_n),e_{k}h_{n-k}\rangle$ could be given in terms of Schr\"oder paths with a given statistic. In \cite{[Wal2019b]} we give a bijection between a subset of Schr\"oder paths with $d$ diagonal steps and the set $\SYT(d+1,1^{n-d-1})\times\{ 1, 2, \ldots, n-d-1\}$. The subset is such that the paths end with a north step and the $\A$ statistic of that path is  equal to $1$.

 \begin{prop}\label{Prop : coro} If  $r=1$ and $\mu\in\{(k,1^{n-k})~|~1\leq k\leq n\}$ or if $r>1$ and $\mu=1^n$ then:
  \begin{align}\label{Eq : nabla}\langle \nabla^r(e_n), s_{\mu}\rangle|_{\hooks}=&\sum_{\tau\in \SYT(\mu)}  s_{r\binom{n}{2}-\m(\tau')}(q,t)+\sum_{i=2}^{\dd(\tau)} s_{r\binom{n}{2}-\m(\tau')-i,1}(q,t).
  \end{align}

Additionally, for all $k$ we have: 
\begin{equation}\label{Eq : delta en}\langle \Delta'_{e_k}(e_n), e_n\rangle|_{\hooks}=\sum_{\tau\in \SYT((n-k,1^k))}  s_{\m(\tau)}(q,t)+\sum_{i=2}^{k} s_{\m(\tau)-i,1}(q,t).
\end{equation}

Furthermore, for all $\mu\in\{(n), (n-1,1), (n-2,1,1),(1^n)\}$, $0\leq k < n-1$, we have:
\begin{equation}\label{Eq : delta mu cacher}e_{n-k-1}^\perp\left(\langle \mathcal{E}_{n,n},s_\mu\rangle|_{\hooks}\right)^{\langle 2\rangle}=\sum_{\tau\in \SYT(\mu)}\sum s_{(k-1+\A(\gamma)-\m(\tau'),1)}(q,t)+\sum s_{(k+\A(\gamma)-\m(\tau'))}(q,t).
\end{equation}
Likewise, if the conjecture $e_k^\perp \langle \mathcal{E}_{n,n}, s_\mu\rangle=_{\langle 2\rangle} \langle\Delta'_{e_{n-k-1}}e_n, s_\mu\rangle$ is true for all $k$, then  Equation \eqref{Eq : nabla} and Equation \eqref{Eq : delta mu cacher} hold for all $\mu$ when $r=1$, and:
\begin{equation}\label{Eq : delta mu}\langle \Delta'_{e_k}(e_n),s_\mu\rangle|_{\textrm{some hooks}}=\sum_{\tau\in \SYT(\mu)}\sum s_{(k-1+\A(\gamma)-\m(\tau'),1)}(q,t)+\sum s_{(k+\A(\gamma)-\m(\tau'))}(q,t)
\end{equation}
where the second sum of Equations \eqref{Eq : delta mu cacher} and Equation \eqref{Eq : delta mu} is over paths in $T_{n,\dd(\tau')}$ of height $k-2$ and $k-1$  and the third sum is over paths in $T_{n,\dd(\tau')}$ of height $k-1$ and $k$.

Finally, if $t=0$ or $q=0$ Equations \eqref{Eq : delta mu cacher} and Equation \eqref{Eq : delta mu} hold for all $\mu$ and Equation \eqref{Eq : nabla} holds for $\mu$ general if $r=1$ and for $r>1$ if $\mu$ is hook shaped. 
\end{prop}
The notation $|_{some ~hooks}$ simply means that some hooks are missing from the sum (i.e. the coefficients of Equation \eqref{Eq : delta mu} constitute a lower bound  for the coefficients of $\langle \Delta'_{e_k}(e_n),s_\mu\rangle|_{\hooks}$).
Note that  Equation \eqref{Eq : nabla} and Equation \eqref{Eq : delta en} hold for all q and t when $\mu=1^n$ and $r$ arbitrary or $r=1$ and $\mu$ is hook shaped.
Additionally, if  $k=n-1$ Equations \eqref{Eq : delta mu cacher} and Equation \eqref{Eq : delta mu} hold if the second sum of is over paths in $T_{n,\dd(\tau')}$ of height $k-2$  and the third sum is over paths in $T_{n,\dd(\tau')}$ of height $k-1$. 

 We will delay the proof of  Proposition \ref{Prop : coro} and Theorem \ref{Thm : main} until after Lemma \ref{Lem : restriction 2 variables}. Before we start let us notice that if $r\not=1$ and $\mu\not=1^n$ computer experimentation suggest that Equation \eqref{Eq : nabla} can be extended even tough its is incomplete (all the terms of the formula seem to appear but some positive terms of $\langle \mathcal{E}_{rn,n}, s_{\mu}\rangle|_{\hooks}$ are missing).

Let us start by proving the statement for the alternants of $\mathcal{E}_{rn,n}$.

 \begin{prop}\label{Prop : alternant rn}For $r=1$ we have:
\begin{equation}\label{Eq : rn}\langle \mathcal{E}_{rn,n}, e_n\rangle|_{\hooks}= \sum_{\gamma \in T_{n}} s_{\hook(\gamma)}
\end{equation}
where $\hook(\gamma)$ is the partition $\left((r-1)\binom{n}{2}+\A(\gamma)+ht(\gamma)+1,1^{n-2-ht(\gamma) }\right)$.
\\
Moreover, for all $r\in \mathbb{N}^*$ such that $ \langle\mathcal{E}_{rn,n}, s_{(j+1),1^{n-j}}\rangle=e^\perp_j \langle\mathcal{E}_{rn,n}, s_{1^{n}}\rangle$ is true for all $j$, Equation \eqref{Eq : rn} holds. 
\end{prop}

\begin{proof}Let $A(q,t):=\psi(\mathcal{E}_{rn,n}, e_n\rangle|_{\hooks})(q,t)$. Recall that F. Bergeron's theorem gives us the equality $ \langle\mathcal{E}_{rn,n}, s_{(j+1),1^{n-j-1}}\rangle=e^\perp_j \langle\mathcal{E}_{rn,n}, s_{1^{n}}.\rangle$ for $r=1$, by Equation \eqref{Eq : 1part r} and Lemma \ref{Lem : altern} we have:
\begin{equation*}A(q,t)=\sum_{j\geq 0}\sum_{k=0}^j (-1)^k \left(q^{r\binom{n}{2}-\binom{j-k+1}{2}}\begin{bmatrix}n-1\\ j-k \end{bmatrix}_{q^{-1}}\right)q^{-k}t^j
\end{equation*}
Furthermore, $s_{(j+1),1^{n-j-1}}$ makes no sense for $j\geq n$ therefore $j$ parses true all value between $0$ and $n-1$ and we can change $j$ for $n-j-1$ and obtain the same result. Then by noticing that $q^{-(j+k)(n-j-k-1)}\begin{bmatrix}n-1\\ j+k \end{bmatrix}_{q}=\begin{bmatrix}n-1\\ j+k \end{bmatrix}_{q^{-1}}$, that $\binom{n}{2}-\binom{n-j-k}{2}-(j+k)(n-j-k-1)=\binom{j+k+1}{2}$ and simplifying we get:
\begin{equation*}A(q,t)=q^{(r-1)\binom{n}{2}}\sum_{j= 0}^{n-1}\sum_{k=0}^{n-j-1} (-1)^k \left(q^{\binom{j+k+1}{2}}\begin{bmatrix}n-1\\ j+k \end{bmatrix}_{q}\right)q^{-k}t^{n-j-1}
\end{equation*}
Considering a twist in variables we obtain:
\begin{equation*}z^{n-2}A(q,z^{-1})=q^{(r-1)\binom{n}{2}}\sum_{j= 0}^{n-1}\sum_{k=0}^{n-j-1} (-1)^k \left(q^{\binom{j+k+1}{2}}\begin{bmatrix}n-1\\ j+k \end{bmatrix}_{q}\right)q^{-k}z^{j-1}
\end{equation*}
Now $\binom{j+k+1}{2}-\binom{j+k}{2}=j+k$ by Proposition \ref{Prop : T_n=sum} (Equation \eqref{Eq : somme alt->pos air+ht}  ) if we set $g_{j}(k)=\binom{j+k+1}{2}$ we get:
\begin{equation*}z^{n-2}A(q,z^{-1})=q^{(r-1)\binom{n}{2}+1}T_n(q,qz)
\end{equation*}
Or equivalently:
\begin{equation*}A(q,t)=t^{n-2}q^{(r-1)\binom{n}{2}+1}T_n(q,qt^{-1})
\end{equation*}
In the generating function $T_{n,0}(q,t)$ the power of the variable $t$ corresponds to the height of the associated path in $T_{n,0}$ and the power of the variable $q$ to the area of the associated path in $T_{n,0}$ we have the interpretation of $t^{n-2}q^{(r-1)\binom{n}{2}+1}T_{n,0}(q,qt^{-1})$ as adding the area, the height and $(r-1)\binom{n}{2}+1$ to the first part and subtracting the height to $n-2$ gives the rest of the hook.
\end{proof}

There is only one standard tableau of shape $1^n$, its conjugate is $(n)$. Since $\dd((n))=0$ and $\m((n))=0$ the formula of the following lemma coincides with the formula of the previous proposition in the case $r=1$ and $\mu=1^n$. 

For the next lemma we need to define $rev_q$. As in \cite{[HRS2018b]} $rev_q:\mathbb{Z}[[X]][q]\rightarrow \mathbb{Z}[[X]][q]$ is defined by $rev_q(f(x)(q))=f(x)(q^{-1})q^{deg_q(f(x)(q))}$.

\begin{lem}\label{Lem : tout mu} If $e_k^\perp \langle \mathcal{E}_{n,n}, s_\mu\rangle=_{\langle 2\rangle} \langle\Delta'_{e_{n-k-1}}e_n, s_\mu\rangle$ is true for all $k$ then: 
\begin{equation}\label{Eq : lift delta}\langle \mathcal{E}_{n,n}, s_{\mu}\rangle|_{\hooks}=\sum_{\tau\in \SYT(\mu)} \sum_{\gamma \in T_{n, \dd(\tau')}} s_{\hook(\gamma)}
\end{equation}
where $\hook(\gamma)$ is the partition $(\A(\gamma)+ht(\gamma)-\m(\tau')+1,1^{n-2-ht(\gamma) })$ and $r=1$.
\end{lem}

\begin{proof}Recall that $\omega$ was defines in Section \ref{Sec : car}. By applying $\omega\circ rev_q$ to corollary 6.13 of \cite{[HRS2018b]} we obtain:

\begin{equation}\label{Eq : HRS}\Delta'_{e_{n-k-1}}(e_n)|_{t=0}=\sum_{\tau\in \SYT(n)} q^{k\dd(\tau')+\binom{n-k}{2}-\m(\tau')}\begin{bmatrix}\dd (\tau)\\ k\end{bmatrix}_q s_{\lambda(\tau)}
\end{equation}

Therefore:
\begin{equation*}f_k^{((n))}(q):=\psi(\langle \Delta'_{e_{n-k-1}}(e_n)|_{t=0},s_\mu\rangle )=\sum_{\tau\in \SYT(\mu)} q^{k\dd(\tau')+\binom{n-k}{2}-\m(\tau')}\begin{bmatrix}\dd (\tau)\\ k\end{bmatrix}_q 
\end{equation*}
and:
\begin{equation*}f_k^{(\mu)}(q)= q^{k(n-1)+\binom{n-k}{2}-\binom{n}{2}}\begin{bmatrix}0\\ k\end{bmatrix}_q 
\end{equation*}

Then by Lemma \ref{Lem : altern} we have:
\begin{equation*}\psi(\langle \mathcal{E}_{n,n}, s_{\mu}\rangle|_{\hooks})=\sum_{j\geq 0}\sum_{k=0}^j (-1)^k \sum_{\tau\in \SYT(\mu)} q^{(j-k)\dd(\tau')+\binom{n-j+k}{2}-\m(\tau')}\begin{bmatrix}\dd (\tau)\\ j-k\end{bmatrix}_q q^{-k}t^j
\end{equation*}

Let $A_\tau^{(\mu)}(q,t)$ be the part of $\psi(\langle \mathcal{E}_{n,n}, s_{\mu}\rangle|_{\hooks})$ when the sum is over $\tau$. In other words we have $\sum_{\tau \in \SYT(\mu)} A_\tau^{(\mu)}(q,t)=\psi(\langle \mathcal{E}_{n,n}, s_{\mu}\rangle|_{\hooks})$. Much like for the previous proposition $j$ parses from $0$ to $\dd(\tau)$. So we can change $j$ for $\dd(\tau)-j$ and obtain the same result, in consequence simplification gives us:
\begin{equation*}
A_\tau^{(\mu)}(q,t)= \sum_{j= 0}^{\dd(\tau)}\sum_{k=0}^{\dd(\tau)-j} (-1)^k q^{(\dd(\tau)-j-k)\dd(\tau')+\binom{n-\dd(\tau)+j+k}{2}-\m(\tau')}\begin{bmatrix}\dd (\tau)\\ j+k\end{bmatrix}_q q^{-k}t^{\dd(\tau)-j}
\end{equation*}
and:
\begin{equation*}
A_\tau^{((n))}(q,t)= \sum_{j= 0}^{0}\sum_{k=0}^{-j} (-1)^k q^{(-j-k)(n-1)+\binom{n+j+k}{2}-\binom{n}{2}}\begin{bmatrix}0 \\ j+k\end{bmatrix}_q q^{-k}t^{-j}=1
\end{equation*}

By recalling that  $n-1-\dd(\tau)=\dd(\tau')$ we notice the equality $\binom{n-\dd(\tau)+j+k}{2}=\binom{j+k+1}{2}+\binom{\dd(\tau')+1}{2}+(j+k)\dd(\tau')$. 
Therefore, by setting $g_j(k)=\dd(\tau)\dd(\tau')+\binom{j+k+1}{2}+\binom{\dd(\tau')+1}{2}-\m(\tau')$  we have $g_j(k)-g_j(k-1)=k+j$. Using Equation \eqref{Eq : alt->pos} of Proposition \ref{Prop : T_n=sum} we obtain:
\begin{equation*}A_\tau^{(\mu)}(q,t)=\sum_{j= 0}^{\dd(\tau)-1} q^{\dd(\tau)\dd(\tau')+\binom{j+1}{2}+\binom{\dd(\tau')+1}{2}-\m(\tau')}\begin{bmatrix}\dd (\tau)-1\\ j\end{bmatrix}_qt^{n-1-\dd(\tau')-j}
\end{equation*}

Has in the last proposition we change variables and obtain:
\begin{equation*}z^{n-2-\dd(\tau')}A_\tau^{(\mu)}(q,z^{-1})=\sum_{j= 0}^{\dd(\tau)-1} q^{\dd(\tau)\dd(\tau')+\binom{j+k+1}{2}+\binom{\dd(\tau')+1}{2}-\m(\tau')}\begin{bmatrix}\dd (\tau)-1\\ j\end{bmatrix}_qz^{j-1}
\end{equation*}

Finally by Proposition \ref{Prop : T_n} we get:
\begin{equation*}z^{n-2-\dd(\tau')}A_\tau^{(\mu)}(q,z^{-1})=q^{1+\dd(\tau)\dd(\tau')+\binom{\dd(\tau')+1}{2}-\m(\tau')}T_{\dd(\tau)+1,0}(q,qz)
\end{equation*}
Hence:
\begin{equation*}A_\tau^{(\mu)}(q,t)=t^{n-2}q^{1-\m(\tau')}q^{(\dd(\tau)-1)\dd(\tau')+\binom{\dd(\tau')+1}{2}}(qt^{-1})^{\dd(\tau')}T_{\dd(\tau)+1,0}(q,qt^{-1})
\end{equation*}
Ergo by Equation \eqref{Eq : Tr0->Tns} of Proposition \ref{Prop : T_n} we have:
\begin{equation*}A_\tau^{(\mu)}(q,t)=t^{n-2}q^{1-\m(\tau')}T_{n,\dd(\tau')}(q,qt^{-1})
\end{equation*}
and:
\begin{equation*}A_\tau^{((n))}(q,t)=1=t^{n-2}q^{1-\binom{n}{2}}q^{\binom{n-1}{2}+n-2}t^{-n+2}=t^{n-2}q^{1-\binom{n}{2}}T_{n,n-1}(q,qt^{-1})
\end{equation*}

\end{proof}
Notice that Equation \eqref{Eq : lift delta} is a lift of $\Delta'_{e_{n-k-1}}(e_n)|_{t=0}$ therefore Equation \eqref{Eq : delta mu cacher} at $t=0$ is a reinterpretation using our object of Equation  \eqref{Eq : HRS} of Haglund, Rhoades and Shimonozo. Furthermore, $\Delta'_{e_{n-1}}(e_n)=\nabla(e_n)$ thus Equation \eqref{Eq : nabla}  at $t=0$ for general $\mu$ and $r=1$ is just a reinterpretation of Equation \eqref{Eq : HRS}

We will now prove that for all $\mu=(d,1^{n-d})$ the restriction to two variables of Equation \eqref{Eq : main} is true independently of F.Bergeron's $e_k^\perp \langle \mathcal{E}_{n,n}, s_\mu\rangle=_{\langle 2\rangle} \langle\Delta'_{e_{n-k-1}}e_n, s_\mu\rangle$ for all $k$ conjecture. In other word the formula gives correctly the coefficients of $s_\mu\otimes s_\lambda$ in $\langle\nabla(e_n),s_\mu\rangle$ when $\mu$ and $\lambda$ are hook shaped.

To this end we recall the following result from \cite{[Wal2019b]}:

\begin{lem}[Wallace]\label{Lem : nabla hookhook} If  $\mu\in\{(d,1^{n-d})~|~1\leq d\leq n\}$ then:
  \begin{align*}\langle \nabla(e_n), s_{\mu}\rangle|_{\hooks}=\sum_{\tau\in \SYT(\mu)}  s_{\m(\tau)}(q,t)+\sum_{i=2}^{\dd(\tau)} s_{\m(\tau)-i,1}(q,t),
  \end{align*}
    \begin{equation*}\langle \nabla^r(e_n), s_\mu \rangle|_{1 \pp}=\sum_{\tau\in \SYT(\mu)}  s_{r\binom{n}{2}-\m(\tau')}(q,0)=\sum_{\tau\in \SYT(\mu)}  s_{r\binom{n}{2}-\m(\tau')}(0,t),
  \end{equation*}
  and:
  \begin{equation*}\langle \nabla^r(e_n), e_n\rangle|_{\hooks}=\sum_{\tau\in \SYT(\mu)}  s_{r\binom{n}{2}-\m(\tau')}(q,t)+\sum_{i=2}^{\dd(\tau)} s_{r\binom{n}{2}-\m(\tau')-i,1}(q,t).
  \end{equation*}
\end{lem}
Note that $\binom{n}{2}-\m(\tau')=\m(\tau)$ and $\dd(\tau)=n-1-\dd(\tau')$.
\begin{lem}\label{Lem : restriction 2 variables}   If  $\mu\in\{(d,1^{n-d})~|~1\leq d\leq n\}$  then:
 \begin{equation*}\left(\sum_{\tau\in \SYT(\mu)} \sum_{\gamma \in T_{n, \dd(\tau')}} s_{\hook(\gamma)}\right)^{\langle 2 \rangle}=\sum_{\tau\in \SYT(\mu)}  s_{r\binom{n}{2}-\m(\tau')}(q,t)+\sum_{i=2}^{\dd(\tau)} s_{r\binom{n}{2}-\m(\tau')-i,1}(q,t),
  \end{equation*}
  where $\hook(\gamma)=\left((r-1)\binom{n}{2}+\A(\gamma)+\h(\gamma)-\m(\tau')+1,1^{n-2-\h(\gamma) }\right)$.
\end{lem}

\begin{proof}Notice that the restriction to two variables is equivalent to the restriction to hooks of length $2$ or less as seen in Section \ref{Sec : car}. We therefore only need to consider the paths of height $n-3$ and $n-2$. The area of a path of height $n-2$ is $\binom{n-1}{2}$ and the hook associated to it has only one part. Furthermore, $(r-1)\binom{n}{2}+\binom{n-1}{2}+n-2+1=r\binom{n}{2}$. This accounts for the first sum on the right hand side. 

The area of a path of height $n-3$ in $T_{n,\dd(\tau')}$ starting with exactly $p$ north steps is $\binom{n-1}{2}-j$, where $j=n-2-p-\dd(\tau')$. The number of north steps at the beginning of the path is bounded by $0\leq p \leq n-3-\dd(\tau')$. This is equivalent to $\dd(\tau)\geq n-1-p -\dd(\tau')\geq 2$. Consequently $(r-1)\binom{n}{2}+\binom{n-1}{2}-j+n-3+1-\m(\tau')=r\binom{n}{2}-\m(\tau')-i$ for $i=n-1-p-\dd(\tau')$ which accounts for the second sum of the right hand side.
\end{proof}

We can now prove Proposition \ref{Prop : coro} and Theorem \ref{Thm : main}.

\begin{proof}[Proof of Theorem \ref{Thm : main}]
For $r=1$ we start with the cases where  $\mu\in\{(n), (n-1,1), (n-2,1,1)\}$. The descent of a standard tableau of shape $(n-k,1^{k})$ has $k$ elements, thus we have $\dd(\tau')=n-1$ if $\tau\in\SYT((n))$, $\dd(\tau')=n-2$ if $\tau\in\SYT((n-1,1))$, $\dd(\tau')=n-3$ if $\tau\in\SYT((n-2,1,1))$. This implies that the height of the paths related to these tableaux must be greater or equal to $n-3$. In consequence $\ell(\hook(\gamma))=\ell(\A(\gamma)+\h(\gamma)-\m(\tau')+1,1^{n-2-\h(\gamma)})\leq 2$. Therefore, no term disappears in the restriction to two variables and  by Lemma \ref{Lem : restriction 2 variables} we have $\sum_{\tau\in \SYT(\mu)} \sum_{\gamma \in T_{n, \dd(\tau')}} s_{\hook(\gamma)}=\langle \nabla(e_n), s_{\mu}\rangle|_{\hooks}$ in these cases. Moreover, it is shown in \cite{[B2018]} that $\langle\mathcal{E}_{n,n}, s_{\mu}\rangle=\langle \nabla(e_n), s_{\mu}\rangle$ when $\mu\in\{(n),(n-1,1), (n-2,1,1), (n-2,2)\}$, which is what we needed.

The remainder of the proof is a direct consequence of Proposition \ref{Prop : alternant rn} and Lemma \ref{Lem : tout mu}.
\end{proof}

\begin{proof}[Proof of Proposition \ref{Prop : coro}] It was proven in  \cite{[H2004]} that $\langle \nabla(e_n),e_{k}h_{n-k}\rangle=\langle \Delta_{e_k} e_n,e_n\rangle$, therefore $\langle \nabla(e_n),s_{(k+1,1^{n-k-1})}\rangle=\langle \Delta'_{e_{n-k-1}} e_n,e_n\rangle$. Hence, Equation \eqref{Eq : nabla} implies the Equation \eqref{Eq : delta en}.

Lemma \ref{Lem : nabla hookhook} proves that Equation \eqref{Eq : nabla} is true for $r=1$ and $\mu\in\{(k,1^{n-k})~|~1\leq k\leq n\}$ or $r>1$ and $\mu=1^n$. It also proves that Equation \eqref{Eq : nabla} holds for $\mu$ hooked shaped when $r\geq 1$ and $t=0$ or $q=0$.

Furthermore, in  Lemma \ref{Lem : tout mu} the formula is constructed by lifting the Schur functions having only one part. Ergo Equation \eqref{Eq : nabla}, Equation \eqref{Eq : delta mu cacher}, and Equation \eqref{Eq : delta mu} holds for all $\mu$ when $r=1$ and $t=0$ or $q=0$ since $\Delta'_{e_{n-1}}(e_n)=\nabla(e_n)$. This also implies that Equation \eqref{Eq : delta mu} holds for all $\mu$ if F.Bergeron's conjecture is true.

Equation \eqref{Eq : delta mu cacher} is the restriction of Theorem \ref{Thm : main} to the paths associated to hooks of length one and two. Equation \eqref{Eq : delta mu} is a rewriting of Equation \eqref{Eq : delta mu cacher}. This is just to state that the paths ending at height $k$ and $k-1$ in $T_{n,s}$ correspond to the Schur functions with one part in $ \Delta'_{e_k}(e_n)$ and the paths ending at height $k-2$ correspond to Schur functions that are indexed by hooks of length two.
\end{proof}

One may be interested to notice that the hook lengths in Theorem \ref{Thm : main} can be computed using only the area statistic and the major index. 


\section{Adjoint dual Pieri rule}\label{Sec : Pieri}

It is shown in \cite{[B2018]} that $ \langle\mathcal{E}_{n,n}, s_{(k+1),1^{n-k-1}}\rangle=e^\perp_k \langle\mathcal{E}_{n,n}, s_{1^{n}}\rangle$. By Proposition \ref{Prop : perp sur chemin} this can be done directly in terms of paths.

 For $k$ such that $0\leq k \leq n-2$ and $\mu=(k+1,1^{n-k-1})$, we consider the following sets of paths:
\begin{equation*}\bigcup_{\tau\in \SYT(\mu)}\{\gamma=N^j\tilde\gamma\in T_{n,\dd(\tau')}~|~j\geq n-k-\textrm{min}( \M(\tau'))\}=T_{n,k}^+
\end{equation*}
 \begin{equation*}\bigcup_{\tau\in \SYT(\mu)}\{\gamma=N^jE\tilde\gamma, \gamma=N^j \in T_{n,\dd(\tau')} ~|~j <n-k- \textrm{min}( \M(\tau'))\}=T_{n,k}^-
 \end{equation*}
 
 \begin{align*} V_{n,k}=&\underset{1\in\DD(\tau')}{\bigcup_{\tau\in \SYT(\mu)}} \left\{\gamma=N^{j}E\tilde\gamma, \gamma=N^j \in T_{n,\dd(\tau')} ~|~\max\{0,n-k-\min(\DD(\tau')\backslash\{1\})\leq j\right\}
\\				&\cup\underset{1\not\in\DD(\tau'), \{n-k+1,\ldots,n-1\}\subset \DD(\tau')}{\bigcup_{\tau\in \SYT(\mu)}} \left\{\gamma=E^{r}\tilde\gamma\in T_{n,\dd(\tau')} ~|~r+1\geq\min(\DD(\tau'))\right\}
\end{align*}
and
\begin{equation*}T_n^k=\{\gamma\in T_{n,0} ~| ~ n-2-\h(\gamma)\geq k\}.
\end{equation*}
\\

 Notice that in for all $\tau\in \SYT(k+1,1^{n-k-1})$, $\dd(\tau')=k$. Hence, all paths of  $T_{n,\dd(\tau')}$ have at most $n-k-2$ north steps. Therefore, if $1\in\DD(\tau')$ then $T_{n,k}^+\cap T_{n,\dd(\tau')}=\emptyset$ and $T_{n,k}^-\cap T_{n,\dd(\tau')}=T_{n,k}^-$. Additionally, one can easily  check that for $\mu=(k+1,1^{n-k-1})$ the sets $T_{n,k}^+$, and $ T_{n,k}^-$ are a partition of the set $\bigcup_{\tau\in \SYT(\mu)}  T_{n,\dd(\tau')}$ and  $V_{n,k}\subset T_{n,k}^-$.
 
   For $k$ between $1$ and $n-2$, we will now define two families $\{\underline{e_{k-}^\perp}\}$ and $\{\underline{e_{k+}^\perp}\}$ of maps:
\begin{equation*} \underline{e_{k-}^\perp}:T_n^{k-1}\backslash\{E^{n-2}\} \rightarrow V_{n,k}\subset T_{n,k}^-~ \text{ and }~ \underline{e_{k+}^\perp}:T_n^{k} \rightarrow T_{n,k}^+
\end{equation*}

For $\gamma\in T_{n,0}$, let us consider the prefix of $\gamma$ ending with the $k$-th east step. The prefix exists by definition of $T_n^k$, since $n-2-\h(\gamma)$ gives the number of east steps. Let  $p_1,\ldots,p_k$ denote the integers such that $p_i$ is the number of north steps before the $i$-th east step. To this we associate $\tau'$, the hook shaped standard tableau such that $\DD(\tau')=\{ n-i-p_i~|~1\leq i\leq k\}$. In consequence $\underline{e_{k+}^\perp}(\gamma)$ is the path in $T_{n,\dd(\tau')}$ given by discarding all the $k$ first east steps of $\gamma$. (See Figure \ref{Fig : perp colonne})

\begin{figure}[h!]
\begin{minipage}{7cm}
\centering
\begin{tikzpicture}[scale=.4]
\filldraw[pink] (-1,0)--(-1,3)--(0,3)--(0,0)--(-1,0);
\filldraw[pink]  (0,0)--(0,1)--(1,1)--(1,2)--(2,2)--(3,2)--(3,3)--(5,3)--(6,3)--(6,2)--(7,2)--(7,1)--(8,1)--(8,0)--(0,0);
\node(h) at (-.5,1.5){$\h$};
\draw[thin, gray] (0,0)--(0,8)--(1,8)--(1,7)--(2,7)--(2,6)--(3,6)--(3,5)--(4,5)--(4,4)--(5,4)--(5,3)--(6,3)--(6,2)--(7,2)--(7,1)--(8,1)--(8,0)--(0,0);
\draw[thin, gray] (1,0)--(1,7);
\draw[thin, gray] (2,0)--(2,6);
\draw[thin, gray] (3,0)--(3,5);
\draw[thin, gray] (4,0)--(4,4);
\draw[thin, gray] (5,0)--(5,3);
\draw[thin, gray] (6,0)--(6,2);
\draw[thin, gray] (7,0)--(7,1);
\draw[thin, gray] (0,1)--(7,1);
\draw[thin, gray] (0,2)--(6,2);
\draw[thin, gray] (0,3)--(5,3);
\draw[thin, gray] (0,4)--(4,4);
\draw[thin, gray] (0,5)--(3,5);
\draw[thin, gray] (0,6)--(2,6);
\draw[thin, gray] (0,7)--(1,7);
\draw[blue, thick]  (0,0)--(0,1)--(1,1)--(1,2)--(2,2)--(3,2)--(3,3)--(5,3);
\draw (2,-1)--(2,9);
\draw (0,-1)--(0,9);
\node(d) at (0,0){\textcolor{blue}{$\bullet$}};
\node(f1) at (8,0){\textcolor{red}{$\bullet$}};
\node(f2) at (7,1){\textcolor{red}{$\bullet$}};
\node(f3) at (6,2){\textcolor{red}{$\bullet$}};
\node(f4) at (5,3){\textcolor{red}{$\bullet$}};
\node(f5) at (4,4){\textcolor{red}{$\bullet$}};
\node(f6) at (3,5){\textcolor{red}{$\bullet$}};
\node(f4) at (2,6){\textcolor{red}{$\bullet$}};
\node(f5) at (1,7){\textcolor{red}{$\bullet$}};
\node(f6) at (0,8){\textcolor{red}{$\bullet$}};
\node(gamma) at (-3,4){$\gamma=$};
\end{tikzpicture}
\end{minipage}
\begin{minipage}{8cm}
\centering
\begin{tikzpicture}[scale=.4]
\filldraw[pink] (-1,2)--(-1,5)--(0,5)--(0,2)--(-1,2);
\node(h) at (-.5,3.5){$\h$};
\filldraw[pink] (0,2)--(0,4)--(1,4)--(1,5)--(4,5)--(4,4)--(5,4)--(5,3)--(6,3)--(6,2)--(7,2)--(7,1)--(8,1)--(8,0)--(7,0)--(7,1)--(5,1)--(5,2)--(0,2);
\draw[thin, gray] (0,2)--(0,8)--(1,8)--(1,7)--(2,7)--(2,6)--(3,6)--(3,5)--(4,5)--(4,4)--(5,4)--(5,3)--(6,3)--(6,2)--(7,2)--(7,1)--(8,1)--(8,0)--(7,0)--(7,1)--(5,1)--(5,2)--(0,2);
\draw[thin, gray] (1,2)--(1,7);
\draw[thin, gray] (2,2)--(2,6);
\draw[thin, gray] (3,2)--(3,5);
\draw[thin, gray] (4,2)--(4,4);
\draw[thin, gray] (5,2)--(5,3);
\draw[thin, gray] (6,1)--(6,2);
\draw[thin, gray] (7,0)--(7,1);
\draw[thin, gray] (5,1)--(7,1);
\draw[thin, gray] (0,2)--(6,2);
\draw[thin, gray] (0,3)--(5,3);
\draw[thin, gray] (0,4)--(4,4);
\draw[thin, gray] (0,5)--(3,5);
\draw[thin, gray] (0,6)--(2,6);
\draw[thin, gray] (0,7)--(1,7);
\draw[blue, thick]  (0,2)--(0,4)--(1,4)--(1,5)--(3,5);
\draw[pink] (-1,0)--(7,0)--(7,1)--(5,1)--(5,2)--(-1,2)--(-1,0);
\node(t) at (2,1){$\m(\tau')$};
\node(d) at (0,2){\textcolor{blue}{$\bullet$}};
\node(f1) at (8,0){\textcolor{red}{$\bullet$}};
\node(f2) at (7,1){\textcolor{red}{$\bullet$}};
\node(f3) at (6,2){\textcolor{red}{$\bullet$}};
\node(f4) at (5,3){\textcolor{red}{$\bullet$}};
\node(f5) at (4,4){\textcolor{red}{$\bullet$}};
\node(f6) at (3,5){\textcolor{red}{$\bullet$}};
\node(f4) at (2,6){\textcolor{red}{$\bullet$}};
\node(f5) at (1,7){\textcolor{red}{$\bullet$}};
\node(f6) at (0,8){\textcolor{red}{$\bullet$}};
\node(gamma) at (-4,4){$\underline{e_{2+}^\perp}(\gamma)=$};
\end{tikzpicture}
\end{minipage}
\caption{The map $\underline{e_{2+}^\perp}$ send the path $\gamma=NENEENEE \in T_{10,0}$ to the path $NNENEE\in T_{10,\dd(\tau')}$, with $\DD(\tau')=\{6,8\}$ }\label{Fig : perp colonne}.
\end{figure}

The map $\underline{e_{k-}^\perp}(\gamma)$ is defined in a similar way. For $\gamma\in T_{n,0}$, let us consider the prefix of $\gamma$ ending with the $k-1$-th east step. We denote by $p_1,\ldots,p_{k-1}$ the integers such that $p_i$ is the number of north steps before the $i$-th east step. Let $h$ be the number of east steps before the first north step. We choose $\tau'$ to be the the hook shaped standard tableau that has the following descent set $\DD(\tau')=\{ n-i-p_i~|~1\leq i\leq k-1\}\cup\{\textrm{max}(1,h-k+2)\}$. Consequently $\underline{e_{k-}^\perp}(\gamma)$ is the path in $T_{n,\dd(\tau')}$ given by discarding all the $k-1$ first east steps of and the first north step of $\gamma$.  (See Figure \ref{Fig : perp 1ere ligne}). Note that we take out the path $E^{n-2}$ because it is associated to the Schur function $s_{1^n}$ and $e_k^\perp(s_{1^n})$ has only one term.

\begin{figure}[h!]
\begin{minipage}{7cm}
\centering
\begin{tikzpicture}[scale=.4]
\filldraw[pink] (-1,0)--(-1,3)--(0,3)--(0,0)--(-1,0);
\filldraw[pink]  (0,0)--(0,1)--(1,1)--(1,2)--(2,2)--(3,2)--(3,3)--(5,3)--(6,3)--(6,2)--(7,2)--(7,1)--(8,1)--(8,0)--(0,0);
\node(h) at (-.5,1.5){$\h$};
\draw[thin, gray] (0,0)--(0,8)--(1,8)--(1,7)--(2,7)--(2,6)--(3,6)--(3,5)--(4,5)--(4,4)--(5,4)--(5,3)--(6,3)--(6,2)--(7,2)--(7,1)--(8,1)--(8,0)--(0,0);
\draw[thin, gray] (1,0)--(1,7);
\draw[thin, gray] (2,0)--(2,6);
\draw[thin, gray] (3,0)--(3,5);
\draw[thin, gray] (4,0)--(4,4);
\draw[thin, gray] (5,0)--(5,3);
\draw[thin, gray] (6,0)--(6,2);
\draw[thin, gray] (7,0)--(7,1);
\draw[thin, gray] (0,1)--(7,1);
\draw[thin, gray] (0,2)--(6,2);
\draw[thin, gray] (0,3)--(5,3);
\draw[thin, gray] (0,4)--(4,4);
\draw[thin, gray] (0,5)--(3,5);
\draw[thin, gray] (0,6)--(2,6);
\draw[thin, gray] (0,7)--(1,7);
\draw[blue, thick]  (0,0)--(0,1)--(1,1)--(1,2)--(2,2)--(3,2)--(3,3)--(5,3);
\draw (1,-1)--(1,9);
\draw (0,-1)--(0,9);
\node(d) at (0,0){\textcolor{blue}{$\bullet$}};
\node(f1) at (8,0){\textcolor{red}{$\bullet$}};
\node(f2) at (7,1){\textcolor{red}{$\bullet$}};
\node(f3) at (6,2){\textcolor{red}{$\bullet$}};
\node(f4) at (5,3){\textcolor{red}{$\bullet$}};
\node(f5) at (4,4){\textcolor{red}{$\bullet$}};
\node(f6) at (3,5){\textcolor{red}{$\bullet$}};
\node(f4) at (2,6){\textcolor{red}{$\bullet$}};
\node(f5) at (1,7){\textcolor{red}{$\bullet$}};
\node(f6) at (0,8){\textcolor{red}{$\bullet$}};
\node(gamma) at (-3,4){$\gamma=$};
\end{tikzpicture}
\end{minipage}
\begin{minipage}{8cm}
\centering
\begin{tikzpicture}[scale=.4]
\filldraw[pink] (-1,2)--(-1,4)--(0,4)--(0,2)--(-1,2);
\node(h) at (-.5,3){$\h$};
\filldraw[pink] (0,1)--(0,3)--(2,3)--(2,4)--(4,4)--(5,4)--(5,3)--(6,3)--(6,2)--(7,2)--(7,1)--(8,1)--(8,0)--(7,0)--(7,1)--(0,1)--(0,2);
\draw[thin, gray] (0,2)--(0,8)--(1,8)--(1,7)--(2,7)--(2,6)--(3,6)--(3,5)--(4,5)--(4,4)--(5,4)--(5,3)--(6,3)--(6,2)--(7,2)--(7,1)--(8,1)--(8,0)--(7,0)--(7,1)--(0,1)--(0,2);
\draw[thin, gray] (1,1)--(1,7);
\draw[thin, gray] (2,1)--(2,6);
\draw[thin, gray] (3,1)--(3,5);
\draw[thin, gray] (4,1)--(4,4);
\draw[thin, gray] (5,1)--(5,3);
\draw[thin, gray] (6,1)--(6,2);
\draw[thin, gray] (7,0)--(7,1);
\draw[thin, gray] (5,1)--(7,1);
\draw[thin, gray] (0,2)--(6,2);
\draw[thin, gray] (0,3)--(5,3);
\draw[thin, gray] (0,4)--(4,4);
\draw[thin, gray] (0,5)--(3,5);
\draw[thin, gray] (0,6)--(2,6);
\draw[thin, gray] (0,7)--(1,7);
\draw[blue, thick]  (0,2)--(0,3)--(2,3)--(2,4)--(4,4);
\draw[pink] (-1,0)--(7,0)--(7,1)--(5,1)--(5,2)--(-1,2)--(-1,0);
\node(t) at (2,.5){$\m(\tau')$};
\node(d) at (0,2){\textcolor{blue}{$\bullet$}};
\node(f1) at (8,0){\textcolor{red}{$\bullet$}};
\node(f2) at (7,1){\textcolor{red}{$\bullet$}};
\node(f3) at (6,2){\textcolor{red}{$\bullet$}};
\node(f4) at (5,3){\textcolor{red}{$\bullet$}};
\node(f5) at (4,4){\textcolor{red}{$\bullet$}};
\node(f6) at (3,5){\textcolor{red}{$\bullet$}};
\node(f4) at (2,6){\textcolor{red}{$\bullet$}};
\node(f5) at (1,7){\textcolor{red}{$\bullet$}};
\node(f6) at (0,8){\textcolor{red}{$\bullet$}};
\node(gamma) at (-4,4){$\underline{e_{2-}^\perp}(\gamma)=$};
\end{tikzpicture}
\end{minipage}
\caption{The map $\underline{e_{2-}^\perp}$ sends the path $\gamma=NENEENEE \in T_{10,0}$ to the path $NEENEE \in T_{10,\dd(\tau')}$, with $\DD(\tau')=\{1,8\}$}.
\label{Fig : perp 1ere ligne}
\end{figure}
We will recall that for $r=1$ the hooks in Theorem \ref{Thm : main} are given by: 
\begin{equation*}\hook(\gamma)=\left(\A(\gamma)+\h(\gamma)-\m(\tau')+1,1^{n-2-\h(\gamma) }\right)
\end{equation*}

\begin{lem}For all $k\leq n-2$ the map $\underline{e_{k+}^\perp}$  is a well defined map.
\end{lem}
\begin{proof}

Let us first notice that hook shaped standard tableaux are uniquely determined by there decent set. Indeed the first column is strictly increasing and all other entries are in the first row. Therefore, an entry, $i$  is in the descent set if and only if $i+1$ is not in the first row.  For fixed $k$, $\gamma\in T_n^k$ has at least $k$ East steps, by definition of $T_n^k$, since $n-2-\h(\gamma)$ is the number of east steps. Additionally, we construct $k$ elements in the descent set. This corresponds to the descent set of a unique tableau of shape $(n-k,k)$. The $p_i$'s are weakly increasing and are subtracted from strictly decreasing numbers hence the $k$ numbers of the descent set that we constructed are all distinct and smaller or equal to $n-1$. Moreover, $ \textrm{min}( \M(\tau'))=n-k-p_k\geq 2$, because $p_k\leq n-2-k$. Finally the constructed path is of length $n-2-k$ with as at least $p_k$ North steps, making it an element of $T_{n,k}^+$.
\end{proof}

\begin{lem}\label{Lem : +}For all $k$ the map $\underline{e_{k+}^\perp}$  is a bijection such that:
\begin{equation*}\hook(\underline{e_{k+}^\perp}(\gamma))=\left(\A(\gamma)+\h(\gamma)+1,1^{n-2-\h(\gamma) -k}\right).
\end{equation*}
\end{lem}

\begin{proof}
Let $\gamma, \beta \in T_n$ such that $\underline{e_{k+}^\perp}(\gamma)=\underline{e_{k+}^\perp}(\beta)$. Let $p_1,\ldots,p_k$ (respectively $r_1\ldots r_k$) be integers such that $p_i$ (respectively $r_i$) gives the number of North steps before the $i$-th East step in $\gamma$ (respectively $\beta$). Because $\underline{e_{k+}^\perp}(\gamma)=\underline{e_{k+}^\perp}(\beta)$ they are associated to the same descent set, ergo the same tableau (this is only true when the tableau is hook shaped). We must have $r_i=p_i$ for all $i$. This mean the paths $\gamma$ and $\beta$ are identical up to the $k$-th east step. But the rest of the paths are also identical as a result of $\underline{e_{k+}^\perp}(\gamma)=\underline{e_{k+}^\perp}(\beta)$. So $\gamma=\beta$ and $\underline{e_{k+}^\perp}$ is an injection.

Given a path $\gamma$ in $T_{n,k}^+$ with associated tableau $\tau$ we can construct $p_1,\ldots, p_k$ by ordering the set $\{n-\pi ~|~ \pi\in\dd(\tau')\}$ and subtracting the $i$-th number by $i$. Since $\gamma\in T_{n,k}^+$ we know $\gamma=N^{p_k}\gamma'$. Hence, we can construct the path $\beta=N^{p_1}EN^{p_2-p_1}E\cdots EN^{p_k-p_{k-1}}E\gamma'$. It is easy to see that $\underline{e_{k+}^\perp}(\beta)=\gamma$ is a consequence of $\underline{e_{k+}^\perp}$ taking out the $k$ first East steps.

Because we start at height $k$ in $T_{n,\dd(\tau')}$, by only erasing East steps we increase the height of $\underline{e_{k+}^\perp}(\gamma)$ by exactly $k$. For this reason we have $\h(\gamma)=\h(\underline{e_{k+}^\perp}(\gamma))-k$. Notice that $\m(\tau')-k$ corresponds to the number of boxes over the part of the path in the first $k$ columns. Therefore, $\A(\gamma)=\A(\underline{e_{k+}^\perp}(\gamma))-\m(\tau')+k$ and we have the claimed equality.
\end{proof}

We obtain a similar result for $\underline{e_{k-}^\perp}$.

\begin{lem}For all $k\leq n-2$ the map $\underline{e_{k-}^\perp}$  is a well defined map.
\end{lem}
\begin{proof}

Has before hook shaped standard tableaux are uniquely determined by there decent set. For fixed $k$, $\gamma\in T_n^{k-1}\backslash \{E^n-2\}$ has at least $k-1$ East steps, by definition. Furthermore, the $k$ numbers of the constructed descent set are smaller or equal than $n-1$.  We need to show that they are all distinct and the $k$ elements in the descent set will corresponds to the descent set of a unique tableau of shape $(k+1,1^{n-k-1})'$. Before we do so, let $\gamma$ be in $T_{n}^{k-1}\backslash \{E^{n-2}\}$. 

If $h\leq k-1$ then $\underline{e_{k-}^\perp}(\gamma)$ is associated to a tableau such that $1\in\DD(\tau')$. Since $\gamma$ as at least $k-1$ East steps by definition of $T_n^{k-1}$ we know that $p_{k-1}\leq n-k-1$. Hence, $\min(\DD(\tau')\backslash\{1\})\geq 2$. The $p_i$'s are weakly increasing and are subtracted from strictly decreasing numbers, in consequence the elements created for are descent set are all distinct. Moreover, $n-k+1- p_{k-1}=\min(\DD(\tau')\backslash\{1\})$  and $h\leq k-1$ which yields $p_{k}-1\geq 0$. So we have $p_k \geq n-k+1-\min(\DD(\tau')\backslash\{1\})$. Thus $\underline{e_{k-}^\perp}(\gamma)$ is in $V_{n,k}$.

If $h>k-1$ then  $p_{k-1}=0$. The height of the path is bounded by the relation $h\leq n-2$ for this reason $n-(k-1)>h-k-2$. Ergo $\underline{e_{k-}^\perp}(\gamma)$ can be associated with a tableau such that $\min(\DD(\tau'))=h-k+2>1$. The $p_i=0$ for all $i$ such that $1\leq i\leq k-1$, thus we have constructed $k$ distinct elements for the descent set.

Furthermore, the path begins with $h-k+1$ east steps by definition of the map. Consequently $\underline{e_{k-}^\perp}(\gamma)$ is in $V_{n,k}$. The erased steps do not depend on the maximum value so one might notice that in the case $1=h-k+2$ we obtain the same tableau and the same path wether we "choose" $1$ or $h-k+2$. Therefore, the map is well defined. 

\end{proof}

\begin{lem}\label{Lem : -}For all $k$ the map $\underline{e_{k-}^\perp}$  is a bijection such that:
\begin{equation*}\hook(\underline{e_{k-}^\perp}(\gamma))=\left(\A(\gamma)+\h(\gamma),1^{n-1-\h(\gamma) -k}\right).
\end{equation*}
\end{lem}

\begin{proof}
We start by showing the map is injective, let $\gamma, \beta \in T_n\backslash\{E^{n-2}\}$ such that $\underline{e_{k-}^\perp}(\gamma)=\underline{e_{k}^\perp}(\beta)$. Let $p_1,\ldots,p_{k-1}$ (respectively $r_1\ldots r_{k-1}$) be integers such that $p_i$ (respectively $r_i$) gives the number of North steps before the $i$-th East step in $\gamma$ (respectively $\beta$). Let $h_\gamma$  (respectively $h_\beta$) be the number of East steps before the first North step in $\gamma$  (respectively $\beta$). It was proven in the previous lemma that $h=\min(\DD(\tau'))+k-2$ or $h=\leq k-1$. Therefore, if $1=\min(\DD(\tau'))$ we have  $h_\gamma \leq k-1$,  $h_\beta \leq k-1$ and the proof is very similar to Lemma \ref{Lem : +}. If $\min(\DD(\tau'))>1$, we have $h_\gamma-k+2=\min(\DD(\tau'))$ and $h_\beta-k+2=\min(\DD(\tau'))$. Ergo $h_\gamma=h_\beta$ and all the $p_i=0$, $1\leq i\leq k-1$. Consequently $\gamma$ and $\beta$ have the same number of East steps before the first North step and the paths after the first North step are the same since $\underline{e_{k-}^\perp}(\gamma)=\underline{e_{k}^\perp}(\beta)$.  Hence, $\gamma=\beta$

We now show that the map is surjective.  Let $\gamma$ be a path in $V_{n,k}$. The set $V_{n,k}$ is a union of sets of paths therefore it can be associated to the tableau $\tau$ corresponding to the set it came from in the union. If $1\in\DD(\tau')$ let $\{1<d_2<\cdots<d_k\}$ be the descent set of $\tau'$ and let $\gamma=N^{j}E\tilde\gamma$ (respectively $\gamma=N^{n-2-k}$). Remember that by definition of $V_{n,k}$ the length of the path is $n-2-k$. Then the path: 
\begin{equation*}
N^{n-1-d_k}EN^{d_k-d_{k-1}-1}EN^{d_{k-1}-d_{k-2}-1}E\cdots N^{d_{3}-d_{2}-1}EN^{j-n+k+d_2}E\tilde\gamma
\end{equation*}
(respectively $N^{n-1-d_k}EN^{d_k-d_{k-1}-1}EN^{d_{k-1}-d_{k-2}-1}E\cdots N^{d_{3}-d_{2}-1}EN^{d_2-2}$) is of length $n-2$ since we only added $k-1$ East steps and $1$ North steps. Hence, $\gamma$ is in $T_n$ because $j\geq n-k-d_2$ (respectively $n-k-2\geq n-k-d_2$) by definition of $V_{n,k}$. Moreover, there are $k$ east steps before $\tilde\gamma$ and $j+1$ north steps (respectively $k-1$ east steps and $n-k-1$ north steps) therefore:
 \begin{equation*}  \underline{e_{k-}^\perp}(N^{n-1-d_k}EN^{d_k-d_{k-1}-1}EN^{d_{k-1}-d_{k-2}-1}E\cdots N^{d_{3}-d_{2}-1}EN^{j-n+k+d_2}E\tilde\gamma)=N^{j}E\tilde\gamma=\gamma\in T_{n,\dd(\tau')}.
 \end{equation*}
 (respectively $\underline{e_{k-}^\perp}(N^{n-1-d_k}EN^{d_k-d_{k-1}-1}EN^{d_{k-1}-d_{k-2}-1}E\cdots N^{d_{3}-d_{2}-1}EN^{d_2-2})=N^{n-k-2}=\gamma$)
 
  If $1\not\in\DD(\tau')$ then $\{d_1<n-k+1<\cdots<n-1\}$ is the descent set of $\tau'$ and let $\gamma=E^{r}\tilde\gamma$. This means the path: 
\begin{equation*}
E^{d_1+k-2}NE ^{r+1-d_1}\tilde\gamma
\end{equation*}
 is in $T_n$ because $r+1\geq d_1$ by definition of $V_{n,k}$. Moreover, there are $r+k-1$ east steps before $\tilde\gamma$ and $1$ north steps therefore:
 \begin{equation*}  \underline{e_{k-}^\perp}(E^{d_1+k-2}NE ^{r+1-d_1}\tilde\gamma)=\gamma\in T_{n,\dd(\tau')}.
 \end{equation*}
Thus $ \underline{e_{k-}^\perp}$ is a bijection.

By erasing the $k-1$ first East steps and the first North step we increase the height of $\underline{e_{k-}^\perp}(\gamma)$ by exactly $k-1$, since we start at height $k$ in $T_{n,\dd(\tau')}$. Hence, $\h(\gamma)=\h(\underline{e_{k-}^\perp}(\gamma))-k+1$. Notice that $\m(\tau')-k+1-\textrm{max}(1,h-k+2)$ corresponds to the number of boxes over the part of the path in the first $k-1$ columns and $\textrm{max}(1,h-k+2)$ correspond to the boxes filled by deleting the first North step. Therefore, $\A(\gamma)=\A(\underline{e_{k-}^\perp}(\gamma))-\m(\tau')+k-1$ and we have the claimed equality.
\end{proof}

Note that one could add conditions to $T_{n,n-1}^+$, $T_{n,n-1}^-$, $V_{n,n-1}$ and $T_n^{n-1}$ to add the case $\mu=(n)$ to the map. The author thinks that it is a lots of commotion just to state that $e_{n-1}^\perp(s^{1^n})=1$.

For the next proposition we extend our maps in the following way $\underline{e_{k+}^\perp}(\gamma)=\emptyset$ if $\gamma\in T_n\backslash T_n^k$, $\underline{e_{k-}^\perp}(\gamma)=\emptyset$ if $\gamma\in T_n\backslash T_n^{k-1}$ and $s_{\hook(\emptyset)}=0$. Observe that $\emptyset$ is not the empty word.
\begin{prop}\label{Prop : perp sur chemin} For all $k\leq n-2$, we have: 
\begin{equation*}e_k^\perp\left(\langle\mathcal{E}_{n,n}, e_n\rangle|_{\hooks}\right)=\sum_{\gamma \in T_{n}} s_{\hook(\underline{e_{k+}^\perp}(\gamma))}+s_{\hook(\underline{e_{k-}^\perp}(\gamma))}.
\end{equation*}
In addition, $\sum_{\tau\in \SYT(k+1,1^{n-k-1})} \sum_{\gamma \in T_{n, \dd(\tau')}} s_{\hook(\gamma)}-e_k^\perp\left(\langle\mathcal{E}_{n,n}, s_{1^{n}}\rangle|_{\hooks}\right)$ has a Schur positive expansion.
\end{prop}

\begin{proof}
By Lemma \ref{Lem : +} and Lemma \ref{Lem : -} we have $e_k^\perp s_{\hook(\gamma)}=s_{\hook(\underline{e_{k+}^\perp}(\gamma))}+s_{\hook(\underline{e_{k-}^\perp}(\gamma))}$. Furthermore, we have the disjoint union $T_{n,k}^-\cup T_{n,k}^+=\bigcup_{\tau\in \SYT((k+1,1^{n-k-1}))} T_{n,k}$ so:
\begin{equation*}\sum_{\tau\in \SYT((k+1,1^{n-k-1}))} \sum_{\gamma \in T_{n, \dd(\tau')}} s_{\hook(\gamma)}=\sum_{\gamma\in T_{n,k}^+}  s_{\hook(\gamma)}+\sum_{\gamma \in T_{n,k}^-}  s_{\hook(\gamma)}
\end{equation*}
 The map $\underline{e_{k+}^\perp}$ (respectively $\underline{e_{k-}^\perp}$) is an injections from $T_n$ into $T_{n,k}^+$ (respectively $T_{n,k}^-$) ergo the result holds. 
\end{proof}

The last proposition gives reason to believe the main theorem holds for all $\mu$, a hook, since the missing terms should be obtained by the restriction to shapes having two columns. If Theorem \ref{Thm : main} is true for all hook shapes $\mu$ then the difference would be given by the equation found in the next lemma.

First we will define $W_{n,k}=T_{n,k}^-\backslash V_{n,k}$:
\begin{align} W_{n,k}=&\label{Eq : 1e union}\underset{1\not\in\DD(\tau')}{\bigcup_{\tau\in \SYT(k+1,1^{n-k-1})}} \left\{\gamma=E^{r}\tilde\gamma\in T_{n,\dd(\tau')} ~|~1<r+1<\min(\DD(\tau'))<n-k\right\}
\\		&\label{Eq : 2e union}\cup\underset{1\not\in \DD(\tau'), \{n-k+1,\ldots,n-1\}\not\subset\DD(\tau')}{\bigcup_{\tau\in \SYT(k+1,1^{n-k-1})}} \left\{\gamma=E^{r}\tilde\gamma\in T_{n,\dd(\tau')} ~|~  r+1\geq\min(\DD(\tau'))\right\}
\\				&\label{Eq : 3e union}\cup\underset{1\in\DD(\tau')}{\bigcup_{\tau\in \SYT(k+1,1^{n-k-1})}} \left\{\gamma=N^{j}E\tilde\gamma\in T_{n,\dd(\tau')} ~|~0\leq j <n-k-\min(\DD(\tau')\backslash\{1\})\right\}
\\				&\label{Eq : 4e union}\cup\underset{1\not\in\DD(\tau')}{\bigcup_{\tau\in \SYT(k+1,1^{n-k-1})}} \left\{\gamma=N^{j}E\tilde\gamma\in T_{n,\dd(\tau')} ~|~0<j<n-k-\min(\DD(\tau'))\}\right\}
\end{align}

One can easily check that these sets are complementary. Note that Proposition \ref{Prop : perp sur chemin} is also a proof of Theorem \ref{Thm : main} for the case $(n-1,n)$ since $W_{n,n-2}=\emptyset$.
\begin{lem}
Let $A=\sum_{\tau\in \SYT(k+1,1^{n-k-1})} \sum_{\gamma \in T_{n, \dd(\tau')}} s_{\hook(\gamma)}-e_k^\perp\left(\langle\mathcal{E}_{n,n}, e_n\rangle|_{\hooks}\right)$ then for $k\geq 2$ we have:
\begin{align}\label{Eq : somme A1}A=&\underset{\min(\DD(\tau'))<n-k}{\underset{1\in\DD(\tau)}{\sum_{\tau\in \SYT(k+1,1^{n-k-1})}}}\sum_{r=1}^{\min(\DD(\tau'))-2}\sum_{\gamma\in T_{n-r,k+1}} s_{\shape1(\gamma)}+\sum_{j=1}^{n-k-1-\min(\DD(\tau'))}\sum_{\gamma\in T_{n-1,j+k}} s_{\shape2(\gamma)}
\\								&\label{Eq : somme A2}+\underset{\{n-k+1,\ldots,n-1\}\not\subseteq \DD(\tau')}{\underset{1\in\DD(\tau)}{\sum_{\tau\in \SYT(k+1,1^{n-k-1})}}}\sum_{r=\min(\DD(\tau'))-1}^{n-k-2}\sum_{\gamma\in T_{n-r,k+1}} s_{\shape1(\gamma)}
\\								&\label{Eq : somme A3}+\underset{1\not\in\DD(\tau)}{\sum_{\tau\in \SYT(k+1,1^{n-k-1})}}\sum_{j=0}^{n-k-1-\min(\DD(\tau')\backslash\{1\})}\sum_{\gamma\in T_{n-1,k+j}} s_{\shape2(\gamma)}
\end{align}

Where $\shape1(\gamma)$ is the partition $\left(\A(\gamma)+\h(\gamma)+1-\m(\tau')+kr,1^{n-2-\h(\gamma)}\right)$, and
$\shape2(\gamma)$ is the partition $\left(\A(\gamma)+\h(\gamma)+1-\m(\tau')+(j+k),1^{n-2-\h(\gamma)}\right)$.
 \\
In particular, for $k=1$  we have:
 \begin{align}\label{Eq : difference} A=& \sum_{m=2}^{n-2}\sum_{r=1}^{m-2}\sum_{\gamma\in T_{n-r,2}} s_{\shape1'(\gamma)}+ \sum_{j=1}^{n-2-m}\sum_{\gamma\in T_{n-1,j+1}} s_{\shape2'(\gamma)}
 \end{align}
Where $\shape1'(\gamma)$ is the partition $\left(\A(\gamma)+\h(\gamma)+1-m+r,1^{n-2-\h(\gamma)}\right)$, and
 $\shape2'(\gamma)$ is the partition $\left(\A(\gamma)+\h(\gamma)+2-m+j,1^{n-2-\h(\gamma)}\right)$.
\end{lem}

\begin{figure}[h!]
\begin{minipage}{8.5cm}
\centering
\begin{tikzpicture}[scale=.5]
\filldraw[gray!50] (0,-2)--(0,1)--(1,1)--(1,-2)--(0,-2);
\filldraw[pink] (1,1)--(1,-2)--(8,-2)--(8,-1)--(7,-1)--(7,0)--(6,0)--(6,1)--(5,1)--(5,2)--(2,2)--(2,1)--(1,1);
\draw[thin, gray] (0,-2)--(0,6)--(1,6)--(1,5)--(2,5)--(2,4)--(3,4)--(3,3)--(4,3)--(4,2)--(5,2)--(5,1)--(6,1)--(6,0)--(7,0)--(7,-1)--(8,-1)--(8,-2)--(0,-2);
\draw[thin, gray] (1,-2)--(1,5);
\draw[thin, gray] (2,-2)--(2,4);
\draw[thin, gray] (3,-2)--(3,3);
\draw[thin, gray] (4,-2)--(4,2);
\draw[thin, gray] (5,-2)--(5,1);
\draw[thin, gray] (6,-2)--(6,0);
\draw[thin, gray] (7,-2)--(7,-1);
\draw[thin, gray] (0,-1)--(7,-1);
\draw[thin, gray] (0,0)--(6,0);
\draw[thin, gray] (0,1)--(5,1);
\draw[thin, gray] (0,2)--(4,2);
\draw[thin, gray] (0,3)--(3,3);
\draw[thin, gray] (0,4)--(2,4);
\draw[thin, gray] (0,5)--(1,5);
\draw[thick, blue] (0,-1)--(0,1)--(1,1);
\draw[thick] (1,1)--(2,1)--(2,2)--(4,2);
\node(fg1) at (0,6){\textcolor{gray}{$\bullet$}};
\node(fg2) at (6,0){\textcolor{gray}{$\bullet$}};
\node(fg3) at (7,-1){\textcolor{gray}{$\bullet$}};
\draw[thick, red] (1,-2)--(1,5)--(2,5)--(2,4)--(3,4)--(3,3)--(4,3)--(4,2)--(5,2)--(5,1)--(6,1)--(6,0)--(7,0)--(7,-1)--(8,-1)--(8,-2)--(1,-2);
\node(dr) at (0,-1){\textcolor{gray}{$\bullet$}};
\node(db) at (1,1){\textcolor{red}{$\bullet$}};
\node(f1) at (2,4){\textcolor{red}{$\bullet$}};
\node(f2) at (3,3){\textcolor{red}{$\bullet$}};
\node(f3) at (4,2){\textcolor{red}{$\bullet$}};
\node(f4) at (5,1){\textcolor{red}{$\bullet$}};
\node(f5) at (1,5){\textcolor{red}{$\bullet$}};
\node(j) at (-1,0){$j$};
\draw[<->] (-.5,-1)--(-.5,1);
\node(k) at (-1,-1.5){$k$};
\draw[<->] (-.5,-2)--(-.5,-1);
\draw[<->] (0,-2.5)--(8,-2.5);
\node(n) at (4,-3){$n-2$};
\node(T) at (7,2){\textcolor{red}{$T_{n-1,k+j}$}};
\node(Tb) at (-2,2){\textcolor{gray}{$T_{n,k}$}};
\end{tikzpicture}
\caption{Comparing the path $NNEENEE=N^2E\tilde\gamma \in T_{10,1}$ and the path $\tilde\gamma\in T_{9,3}$. Same height $\h(N^2E\tilde\gamma)=\h(\tilde\gamma)$ though $\A(N^2E\tilde\gamma)=\A(\tilde\gamma)+3$.}
\label{Fig : commence pas nord}
\end{minipage}
\begin{minipage}{8.5cm}
\begin{tikzpicture}[scale=.5]
\filldraw[gray!50] (0,-2)--(0,-1)--(3,-1)--(3,-2)--(0,-2);
\filldraw[pink] (3,0)--(3,-2)--(8,-2)--(8,-1)--(7,-1)--(7,0)--(6,0)--(6,1)--(5,1)--(5,0)--(3,0);
\draw[thin, gray] (0,-2)--(0,6)--(1,6)--(1,5)--(2,5)--(2,4)--(3,4)--(3,3)--(4,3)--(4,2)--(5,2)--(5,1)--(6,1)--(6,0)--(7,0)--(7,-1)--(8,-1)--(8,-2)--(0,-2);
\draw[thin, gray] (1,-2)--(1,5);
\draw[thin, gray] (2,-2)--(2,4);
\draw[thin, gray] (3,-2)--(3,3);
\draw[thin, gray] (4,-2)--(4,2);
\draw[thin, gray] (5,-2)--(5,1);
\draw[thin, gray] (6,-2)--(6,0);
\draw[thin, gray] (7,-2)--(7,-1);
\draw[thin, gray] (0,-1)--(7,-1);
\draw[thin, gray] (0,0)--(6,0);
\draw[thin, gray] (0,1)--(5,1);
\draw[thin, gray] (0,2)--(4,2);
\draw[thin, gray] (0,3)--(3,3);
\draw[thin, gray] (0,4)--(2,4);
\draw[thin, gray] (0,5)--(1,5);
\draw[thick] (3,0)--(5,0)--(5,1);
\node(fg1) at (0,6){\textcolor{gray}{$\bullet$}};
\node(fg2) at (6,0){\textcolor{gray}{$\bullet$}};
\node(fg3) at (7,-1){\textcolor{gray}{$\bullet$}};
\draw[thick, red] (3,-2)--(3,3)--(4,3)--(4,2)--(5,2)--(5,1)--(6,1)--(6,0)--(7,0)--(7,-1)--(8,-1)--(8,-2)--(3,-2);
\draw[thick, blue] (0,-1)--(3,-1)--(3,0);
\node(dr) at (0,-1){\textcolor{gray}{$\bullet$}};
\node(db) at (3,0){\textcolor{red}{$\bullet$}};
\node(f2) at (3,3){\textcolor{red}{$\bullet$}};
\node(f3) at (4,2){\textcolor{red}{$\bullet$}};
\node(f4) at (5,1){\textcolor{red}{$\bullet$}};
\node(f1) at (2,4){\textcolor{gray}{$\bullet$}};
\node(f5) at (1,5){\textcolor{gray}{$\bullet$}};
\node(r) at (1.5,-.3){$r$};
\draw[<->] (0,-.7)--(3,-.7);
\node(k) at (-1,-1.5){$k$};
\draw[<->] (-.5,-2)--(-.5,-1);
\draw[<->] (0,-2.5)--(8,-2.5);
\node(n) at (4,-3){$n-2$};
\node(T) at (7,2){\textcolor{red}{$T_{n-r,k+1}$}};
\node(Tb) at (-2,2){\textcolor{gray}{$T_{n,k}$}};
\end{tikzpicture}
\caption{Comparing the path $EEENEEN=E^3N\tilde\gamma \in T_{10,1}$ and the path $\tilde\gamma\in T_{7,2}$. Same height $\h(E^3N\tilde\gamma)=\h(\tilde\gamma)$ though $\A(E^3N\tilde\gamma)=\A(\tilde\gamma)+3$.}
\label{Fig : commence pas est}
\end{minipage}
\end{figure}

\begin{proof}We know that $T_n^k=T_{n,k}^+\cup T_{n,k}^-$, so by Proposition \ref{Prop : perp sur chemin}, Lemma \ref{Lem : +} and Lemma \ref{Lem : -} we have:
\begin{align*}
A&=\sum_{\gamma\in T_{n,k}^+}s_{\hook(\gamma)}+\sum_{\gamma\in T_{n,k}^-}s_{\hook(\gamma)}-\sum_{\gamma\in T_{n,k}^+}s_{\hook(\gamma)}-\sum_{\gamma\in V_{n,k}}s_{\hook(\gamma)},
\\	&=\sum_{\gamma\in T_{n,k}^-}s_{\hook(\gamma)}-\sum_{\gamma\in V_{n,k}}s_{\hook(\gamma)},
\\	&=\sum_{\gamma\in W_{n,k}}s_{\hook(\gamma)}.
\end{align*}
The last equality is a consequence of $W_{n,k}=T_{n,k}^-\backslash V_{n,k}$. Up to a slight change in the area statistic, the paths $\gamma\in T_{n,k}$ such that $\gamma=N^jE\tilde\gamma$ are the same as the paths $\tilde\gamma \in T_{n-1,j+k}$ (see Figure \ref{Fig : commence pas nord}). To obtain the same hook shape we only need to add $j+k$ to the area. Hence, the Sum of Line \eqref{Eq : somme A3} corresponds to the set in Line \eqref{Eq : 3e union} and the second sum of Line \eqref{Eq : somme A1} to the set in Line \eqref{Eq : 4e union}.

Similarly, the paths $\gamma\in T_{n,k}$ such that $\gamma=E^rN\tilde\gamma$ are the same as the paths $\tilde\gamma \in T_{n-r,k+1}$ (see Figure \ref{Fig : commence pas est}). To obtain the same hook shape we only need to add $rk$ to the area. Consequently the first Sum of Line \eqref{Eq : somme A1} corresponds to the set in Line \eqref{Eq : 1e union} and the sum of Line \eqref{Eq : somme A2} to the set in line \eqref{Eq : 2e union}. Note that for Line \eqref{Eq : somme A1} and  Line \eqref{Eq : somme A2} the case $r=n-k-2$ corresponds to the $T_{k+2,k+1}$ which correspond to the paths $E^{n-2-k}$ in $T_{n,k}$ that as area $k(n-k-2)$ greater than $\A(\epsilon)$, where $\epsilon$ is the only path of  $T_{k+2,k+1}$. This works with the convention $T_{k+2,k+1}=T_{k+2,k}$.

For the restriction to $k=1$ one only needs to notice that for all tableaux $\tau$ in $\SYT(2,1^{n-2})$ the descent set contains only one element. Therefore, two of the sums are empty and the result follows.
\end{proof}


\section{Bijections and starting the second column}\label{Sec : 2e colonne}
If we dismiss the first part we can see hook shaped partitions as partitions with one column. We don't have a formula for the restriction to partitions that have two columns, but we can start the second column. Before we prove our formula for the restriction to shapes $\{(a,2,1^k)~|~ k\in\mathbb{N}, a\in \mathbb{N}_{\geq 2}\}$ we need  preliminary result.
\\

Let $T_{n,0,h}^E$ be the subset of paths of $T_{n,0}$ that start with an east step and have height $h$. Let $\SYT(k+1,1^{n-k-1})_{S}$ be the set of tableaux of shape $(k+1,1^{n-k-1})$ for which the descent set contains the set $S$. For a path $\gamma$ let $n_i$ be the number of east steps before the $i$-th north step. 

Define $\Phi_k: T_{n,0,n-k-3}^E \rightarrow  \SYT(k+1,1^{n-k-1})_{\{1,2\}}$ by $\Phi_k(\gamma)$ is the unique hook shape tableau having $\{(n-i-n_i+1)| 1\leq i \leq \h(\gamma)\}\cup\{1,2\}$ as a descent set (see Figure \ref{Fig : Phi}). 

\begin{figure}[h!]
\begin{minipage}{8.5cm}
\centering
\begin{tikzpicture}[scale=.5]
\draw[gray, thin] (0,0)--(0,5)--(1,5)--(1,4)--(2,4)--(2,3)--(3,3)--(3,2)--(4,2)--(4,1)--(5,1)--(5,0)--(0,0);
\draw[gray, thin] (0,1)--(4,1);
\draw[gray, thin] (0,2)--(3,2);
\draw[gray, thin] (0,3)--(2,3);
\draw[gray, thin] (0,4)--(1,4);
\draw[gray, thin] (1,0)--(1,4);
\draw[gray, thin] (2,0)--(2,3);
\draw[gray, thin] (3,0)--(3,2);
\draw[gray, thin] (4,0)--(4,1);
\draw[red,thick] (0,0)--(1,0)--(1,2)--(2,2)--(2,3);
\node(d) at (0,0){\textcolor{red}{$\bullet$}};
\node(f) at (2,3){\textcolor{red}{$\bullet$}};
\node(3) at (6,2.7){$3$};
\node(5) at (6,1.65){$5$};
\node(6) at (6,0.5){$6$};
\node(m) at (8,2){$\mapsto$};
\draw (10,0)--(10,6)--(11,6)--(11,1)--(12,1)--(12,0)--(10,0);
\draw (10,1)--(11,1);
\draw (10,2)--(11,2);
\draw (10,3)--(11,3);
\draw (10,4)--(11,4);
\draw (10,5)--(11,5);
\draw (11,0)--(11,1);
\node(1) at (10.5,.5){\textcolor{blue}{$1$}};
\node(2) at (10.5,1.5){\textcolor{blue}{$2$}};
\node(3) at (10.5,2.5){\textcolor{red}{$3$}};
\node(4) at (10.5,3.5){$4$};
\node(5) at (11.5,.5){\textcolor{red}{$5$}};
\node(6) at (10.5,4.5){\textcolor{red}{$6$}};
\node(7) at (10.5,5.5){$7$};
\node(des) at (6,-1){$\DD(\Phi_1(ENNEN))=\{1,2,3,5,6\}$};
\end{tikzpicture}
\caption{An example of the map $\Phi_1$ when $n=7$ and $k=1$. With $ENNEN\in T_{7,0,3}^E$.}
\label{Fig : Phi}
\end{minipage}
\begin{minipage}{8.5cm}
\centering
\begin{tikzpicture}[scale=.5]
\draw[gray, thin] (0,0)--(0,5)--(1,5)--(1,4)--(2,4)--(2,3)--(3,3)--(3,2)--(4,2)--(4,1)--(5,1)--(5,0)--(0,0);
\draw[gray, thin] (0,1)--(4,1);
\draw[gray, thin] (0,2)--(3,2);
\draw[gray, thin] (0,3)--(2,3);
\draw[gray, thin] (0,4)--(1,4);
\draw[gray, thin] (1,0)--(1,4);
\draw[gray, thin] (2,0)--(2,3);
\draw[gray, thin] (3,0)--(3,2);
\draw[gray, thin] (4,0)--(4,1);
\draw[red,thick] (0,0)--(0,2)--(2,2)--(2,3);
\node(d) at (0,0){\textcolor{red}{$\bullet$}};
\node(f) at (2,3){\textcolor{red}{$\bullet$}};
\node(2) at (6,2.7){$2$};
\node(5) at (6,1.65){$5$};
\node(6) at (6,0.5){$6$};
\node(m) at (8,2){$\mapsto$};
\draw (10,0)--(10,6)--(11,6)--(11,1)--(12,1)--(12,0)--(10,0);
\draw (10,1)--(11,1);
\draw (10,2)--(11,2);
\draw (10,3)--(11,3);
\draw (10,4)--(11,4);
\draw (10,5)--(11,5);
\draw (11,0)--(11,1);
\node(1) at (10.5,.5){\textcolor{blue}{$1$}};
\node(2) at (10.5,1.5){\textcolor{red}{$2$}};
\node(3) at (10.5,2.5){\textcolor{blue}{$3$}};
\node(4) at (10.5,3.5){$4$};
\node(5) at (11.5,.5){\textcolor{red}{$5$}};
\node(6) at (10.5,4.5){\textcolor{red}{$6$}};
\node(7) at (10.5,5.5){$7$};
\node(des) at (6,-1){$\DD(\Omega_1^1(NNEEN))=\{1,2,3,5,6\}$};
\end{tikzpicture}
\caption{An example of the map $\Omega_1^1$ when $n=7$, $k=1$ and $j=1$. With $NNEEN \in T_{7,0,3}^{(1)}$.}
\label{Fig : Omega_j^k}
\end{minipage}
\end{figure}

\begin{lem}\label{Lem : Phi}The map $\Phi_k$ is a well defined bijective maps. Additionally, for $\gamma\in T_{n,0}^E$ we have:
\begin{equation*} \A(\gamma)+\h(\gamma)+1=\m(\Phi_k(\gamma))-\dd(\Phi_k(\gamma))
\end{equation*}
In particular the image of $\Phi_k$ is the set of tableaux of shape $(k+1,1^{n-k-1})$ such that the descent set contains the set $\{1,2\}$.
\end{lem}
\begin{proof} For all hook shaped tableaux the descent set is given by subtracting one to each entry that does not lie in the first row. Because the remainder of the entries are in the first row the descent set uniquely determines a hook shaped tableau. Moreover, a path, $\gamma$ of $T_{n,0,n-k-3}^E$ starts with one East step and has exactly $k+1$ east steps. Ergo for all $i$ such that $1\leq i\leq n-k-3$ we have $1\leq n_i\leq k+1$. This yields $n-i+1 > n-i-n_i+1  \geq 3$. For all path $\gamma$, the set $\{n_1,\ldots, n_{n-k-3}\}$ is an increasing sequence of positive integers therefore the $n-k-3$ elements of the descent set created are all distinct values of $\{1,\ldots,n-1\}$. Hence, $\Phi_k$ associates $\gamma$ to a  unique tableau of shape $(k+1,1^{n-k-1})$ consequently the maps, $\Phi_k$ is well defined. 

Let $\gamma, \pi\in T_{n,0}^E$ be such that $\Phi_k(\gamma)=\Phi_k(\pi)$. By the previous paragraph the height of the path determines the shape of the tableau thus $\gamma$ and $\pi$ are of same height. Let:
\begin{equation}\label{Eq : descent set} \{1 <2 < d_1 < d_2 < \cdots < d_{n-k-3}\}=\DD(\Phi_k(\gamma))=\DD(\Phi_k(\pi)).
\end{equation}
By construction we must have $d_{n-k-i-2}=n-i-n_i+1$ and therefore we know the number of east step before each north step which uniquely determines a path. Ergo $\gamma=\pi$. 

Now we show that the map is surjective. Let $\tau\in  \SYT(k+1,1^{n-k-1})_{\{1,2\}}$ with $ \DD(\tau)=\{1 <2 < d_1 < d_2 < \cdots < d_{n-k-3}\}$. Then the path:
\begin{equation*}\label{Eq : chemin}\gamma=E^{n-d_{n-k-3}}NE^{d_{n-k-3}-d_{n-k-4}-1}N\cdots E^{d_{2}-d_{1}-1}NE^{d_1-3},
\end{equation*}
is in $T_{n,0,n-k-3}^E$ since $d_{n-k-3}\leq n-1$ and the path has $n-k-3$ North steps and $n-2$ steps. Moreover, $\Phi_k(\gamma)=\tau$.

Finally the area of a path $\gamma$ is equal to $\sum_{i=1}^{n-k-3}n-i-1-n_i$, $\m(\Phi_k(\gamma))=3+\sum_{i=1}^{n-k-3}n-i-n_i+1$ and $\DD(\Phi_k(\gamma))=n-k-1$ which yields:
\begin{align*}\h(\gamma)+1+ \A(\gamma)&=n-k-2+\sum_{i=1}^{n-k-3}(n-i-n_i-1)
\\							&= - n+k+4+\sum_{i=1}^{n-k-3}(n-i-n_i+1)
\\							&= -( \DD(\Phi_k(\gamma))-2)+1+(\m(\Phi_k(\gamma))-3)
\\							&=\m(\Phi_k(\gamma)) -\DD(\Phi_k(\gamma))
\end{align*}
\end{proof}

Let $T_{n,0,h}^{(j)}$ be the paths of $T_{n,0}$ that start with a north step, end with exactly $j$ north steps and has height $h$.

For $0\leq k\leq n-3$, we also define $\Omega_k^j:  T_{n,0,n-k-3}^{(j)} \rightarrow \SYT(k+1,1^{n-k-1})_{\{1,\ldots, j+2 , n-1\}}$ by $\Omega_k^j(\gamma)$ is the unique hook shape tableau having $\{(n-i-n_i)| 1\leq i \leq \h(\gamma)\}\cup\{1,j+2\}$ as a descent set (see Figure \ref{Fig : Omega_j^k}). As before, in a path $\gamma$ the $n_i$'s are the number of east steps before the $i$-th north step. Note that $0\leq j\leq n-k-3$ since  the path can not have more north steps than it is height. 

\begin{lem}\label{Lem : Omega} The maps $\Omega_k^j$ are well defined bijective maps. Additionally, for $\gamma \in T_{n,0,n-k-3}^{(j)}$ we have:
\begin{equation*}\A(\gamma)+\h(\gamma)+1=\m(\Omega_k^j(\gamma))-(j+2)
\end{equation*}
Furthermore the image of $\Omega_k^j$ is the set of all tableaux, $\tau$, of shape $(k+1,1^{n-k-1})$, such that $\{1,2, \ldots, j+2, n-1\}\subseteq \DD(\tau)$.
\end{lem}
\begin{proof} As before, the descent set uniquely determines a hook shaped tableau. Furthermore, a path of height $n-k-3$ has exactly $k+1$ east steps. Hence, for all $i$, $0\leq n_i\leq k+1$.  Which yields $n-i\geq n-i-n_i \geq n-k-i-1\geq 2$ as a result of $i\leq n-k-3$. In particular the paths start with a north step ergo $n_1=0$ and $n-1\in\DD(\Omega_k^j(\gamma))$. Additionally, for all $i$ such that $\h(\gamma)-j+1 \leq i \leq \h(\gamma)$ the value of $n_i$ is $k+1$, since the paths end with exactly $j$ North steps. In consequence we have $j+1\geq n-i-n_i=n-k-1-i \geq 2$.  For all path $\gamma$, the set $\{n_1,\ldots, n_{\h(\gamma)}\}$ is an increasing sequence of positive integers thus $n-i-n_i>n-(i+1)-n_{i+1}$ and the $n-k-3$ elements of the descent set created are all distinct. We then have $j$ distinct values between $2$ and $j+1$ for this reason $\{2, \ldots, j+1, n-1\}$ is a subset of $\DD(\Omega_k^j(\gamma))$. Finally $n_{h(\gamma)-j}>k+1$, by definition of $T_{n,0,n-k-3}^{(j)}$. Therefore, $j+2$ is not one of the $n-k-3$ elements created and the map $\Omega_k^j$ associates $\gamma$ to a  unique tableau of shape $(k+1,1^{n-k-1})$. Consequently it is well defined. 

To show the maps are bijective we can use the same proof as in Lemma \ref{Lem : Phi} if we swap the descent set on Line  \eqref{Eq : descent set} for $\{1  <  d_1<\cdots <d_j<j+2< d_{j+1} < \cdots < d_{\h(\gamma)-1}<n-1\}$ and the path on Line \eqref{Eq : chemin} for $NE^{n-d_{n-k-4}-2}NE^{d_{n-k-4}-d_{n-k-5}-1}N\cdots E^{d_{j+2}-d_{j+1}-1}NE^{d_{j+1}-j-2}N^j$.

The area of a path $\gamma$ is equal to $\sum_{i=1}^{\h(\gamma)}n-i-n_i-1$, $\m(\Omega_k^j(\gamma))=3+j+\sum_{i=1}^{\h(\gamma)}n-i-n_i$ Ergo:
\begin{align*} \A(\gamma)+\h(\gamma)+1&=\h(\gamma)+1+\sum_{i=1}^{\h(\gamma)}(n-i-1-n_i)
\\							&=1+\sum_{i=1}^{\h(\gamma)}(n-i-n_i) 
\\							&=\m(\Omega_k^j(\gamma)) -(j+2)
\end{align*}
\end{proof}
 
Notice that $T_{n,0}^E \cup \bigcup_{j,k} T_{n,0,n-k-3}^{(j)}=T_{n,0}$. By lifting the formula for hook shapes in two variables we get a first formula for the alternant restricted to the shape $R=\{(a,2,1^k)~|~ k\in\mathbb{N}, a\in \mathbb{N}_{\geq 2}\}$. 
\begin{prop} 
For $\mathcal{R}=\{(a,2,1^k)~|~ k\in\mathbb{N}, a\in \mathbb{N}_{\geq 2}\}$, we have:
\begin{align}\langle \mathcal{E}_{n,n},e_n\rangle|_{\mathcal{R}}=&\label{Eq : 2 colone lift} \sum_{k=1}^{n-4}\sum_{i=2}^{n-k-2}\underset{\{1,\ldots,i, n-1\}\not\subseteq \DD(\tau)}{\underset{1\in \DD(\tau)}{\sum_{\tau\in \SYT(k+1,1^{n-k-1})}}}s_{\shape(\gamma)}
\\			=&\sum_{k=1}^{n-4}\sum_{i=2}^{n-k-2}\underset{\{1,\ldots,i\}\not\subseteq \DD(\tau)}{\underset{1\in \DD(\tau)}{\sum_{\tau\in \SYT(k+1,1^{n-k-1})}}}s_{\shape(\gamma)}+\underset{\{1,\ldots,i\}\subseteq \DD(\tau)}{\underset{n-1\not\in \DD(\tau)}{\sum_{\tau\in \SYT(k+1,1^{n-k-1})}}}s_{\shape(\gamma)}
\end{align}
where $\shape(\gamma)$ gives the partition $\left( \m(\tau)-i, 2, 1^{k-1}\right)$.
\end{prop}

\begin{proof} 
Let $h_k(q)=\psi\left(\sum_{d=0}^{k} (-1)^{k-d}( e_d^\perp \langle \mathcal{E}_{n,n},e_n\rangle |_{\mathcal{V}_{1}})^{\langle 2\rangle}\right)q^{k-d}t^{-1}$, in respect to the notations in Lemma \ref{Lem : inclu-exclu}. Note that $|_{\mathcal{V}_{1}}$ is the same as $|_{\hooks}$. According to \cite{[Wal2019b]}, we have:

\begin{equation*}h_k=\underset{\{1,2\}\subseteq \DD(\tau)}{\sum_{\tau \in \SYT(k+1,1^{n-k-1})}} q^{\m(\tau)-\dd(\tau)}
+\underset{1\in \DD(\tau)}{\sum_{\tau \in \SYT(k+1,1^{n-k-1})}} \sum_{i=2}^{n-k-2} q^{\m(\tau)-i}
\end{equation*}

Let $ g_k(q)=\psi\left((\sum_{d=0}^{k} (-1)^{k-d}e_d^\perp \langle \mathcal{E}_{n,n}, e_n\rangle |_{\mathcal{V}_{1}})^{\langle 2\rangle}|_{\mathcal{V}_{1}}\right)q^{k-d}t^{-1}$. Since there is only one tableau of shape $(n)$ and it has an empty descent set, we already know from Equation \eqref{Eq : delta mu cacher} that:
\begin{align*} g_k(q)&=\sum_{d=0}^{k} (-1)^{k-d}  \underset{\h(\gamma)=n-d-2}{\sum_{\gamma \in T_{n,0}}} q^{\A(\gamma)+\h(\gamma)+1+k-d}+\underset{\h(\gamma)=n-d-3}{\sum_{\gamma \in T_{n,0}}} q^{\A(\gamma)+\h(\gamma)+1+k-d}
\\			&=\underset{\h(\gamma)=n-k-3}{\sum_{\gamma \in T_{n,0}}} q^{\A(\gamma)+\h(\gamma)+1}
\end{align*}

Recall the remark under Equation  \eqref{Eq : delta mu cacher} states that if $d=0$ we do not account for the paths of height $n-2$. Notice $k$ is only defined for $0\leq k \leq n-3$. More over:
\begin{equation*}(e^\perp_k(\langle \mathcal{E}_{n,n},e_n\rangle |_{\mathcal{V}_{2}})^{\langle 2\rangle})|_{\mathcal{V}_{1}}=(e^\perp_k\langle \mathcal{E}_{n,n},e_n\rangle )^{\langle 2\rangle}|_{\mathcal{V}_{1}}-(e^\perp_k (\langle \mathcal{E}_{n,n},e_n\rangle |_{\mathcal{V}_{1}})^{\langle 2\rangle})|_{\mathcal{V}_{1}}
\end{equation*}
Then by Lemma \ref{Lem : inclu-exclu} we have:
  \begin{equation*} 
\Psi(\langle \mathcal{E}_{n,n},e_n\rangle|_{\mathcal{V}{2}})(q,t)=\sum_{k=0}^{n-3} (h_k(q)-g_k(q))t^k.
   \end{equation*}
   Since for the sets $T_{n,0,n-k-3}^{(j)} $, $ T_{n,0,n-k-3}^E$ partitions $T_{n,0,n-k-3}$, we can apply the maps $\Phi_k$ and $\Omega_k^{i-2}$ on the paths of $g_k$. Ergo by Lemma \ref{Lem : Phi} and Lemma \ref{Lem : Omega} we can cancel out all the negative terms and obtain the result has stated. 
\end{proof}

Before we give a simpler interpretation in terms of $T_{n,k}$ we need a new map. Let $\beta_d$  be a family of maps, $\beta_d: \SYT(d+1,1^{n-d-1})_{\{1\}}\rightarrow T_{n,0,n-d-2}$ defined, for $\DD(\tau)=\{1<r_2\cdots <r_{n-d-1}\}$, by:
\begin{equation*}
\beta_d(\tau)=E^{n-1-r_{n-d-1}}NE^{r_{n-d-1}-r_{n-d-2}-1}N\cdots NE^{r_{n-d-i+1}-r_{n-d-i}-1}N\cdots E^{r3-r_2-1}NE^{r_{2}-2}. 
\end{equation*}
(See Figure \ref{Fig : ex beta} for an example.)

\begin{figure}[h!]
\centering
\begin{tikzpicture}[scale=.5]
\draw[gray, thin] (8,0)--(8,5)--(9,5)--(9,4)--(10,4)--(10,3)--(11,3)--(11,2)--(12,2)--(12,1)--(13,1)--(13,0)--(8,0);
\draw[gray, thin] (8,1)--(12,1);
\draw[gray, thin] (8,2)--(11,2);
\draw[gray, thin] (8,3)--(10,3);
\draw[gray, thin] (8,4)--(9,4);
\draw[gray, thin] (9,0)--(9,4);
\draw[gray, thin] (10,0)--(10,3);
\draw[gray, thin] (11,0)--(11,2);
\draw[gray, thin] (12,0)--(12,1);
\draw[red,thick] (8,0)--(9,0)--(9,2)--(10,2)--(10,3);
\node(d) at (8,0){\textcolor{red}{$\bullet$}};
\node(f) at (10,3){\textcolor{red}{$\bullet$}};
\node(2) at (14,2.7){$2$};
\node(5) at (14,1.65){$4$};
\node(6) at (14,0.5){$5$};
\node(tau) at (-2,2){$\tau=$};
\node(m) at (6,2){$\mapsto$};
\draw (0,0)--(0,5)--(1,5)--(1,1)--(3,1)--(3,0)--(0,0);
\draw (0,1)--(1,1);
\draw (0,2)--(1,2);
\draw (0,3)--(1,3);
\draw (0,4)--(1,4);
\draw (2,0)--(2,1);
\draw (1,0)--(1,1);
\node(1) at (0.5,.5){\textcolor{blue}{$1$}};
\node(2) at (0.5,1.5){\textcolor{red}{$2$}};
\node(3) at (0.5,2.5){$3$};
\node(5) at (0.5,3.5){\textcolor{red}{$5$}};
\node(4) at (1.5,.5){\textcolor{red}{$4$}};
\node(6) at (0.5,4.5){$6$};
\node(7) at (2.5,.5){$7$};
\node(des) at (0,-1){$\DD(\tau)=\{1,2,4,5\}$};
\node(beta) at (10,-1){$\beta_2(\tau)=ENNEN$};
\end{tikzpicture}
\caption{An example of the map $\beta_2$ when $n=7$ and $d=2$. The image $\beta_2(\tau)$ is in $T_{7,0,3}$.}
\label{Fig : ex beta}
\end{figure}

\begin{lem}\label{Lem : bij descent}
The maps $\beta_d$ are well defined bijections. Moreover, for  all $\tau$ we have $\m(\tau)=\A(\beta_d(\tau))+\h(\beta_d(\tau))+1$.
\end{lem}
\begin{proof}For all $\tau$ we have the image by $\beta_d$ is a path with $n-d-2$ north steps by construction. So it as height $n-d-2$. The number of steps of the path is:
\begin{align*}
(n-d-2)+(n-1-r_{n-d-1})+(r_2-2)+\sum_{i=1}^{n-d-3}r_{n-d-i}-r_{n-d-i-1}-1=n-2
\end{align*}
So the image is a path of $T_{n,0,n-d-2}$. Given $\beta_d(\tau)=\beta_d(\pi)$ one could reverse engineer and find $\DD(\tau)=\DD(\pi)$. Conversely, the staircase shape of the grid $T_{n,0}$ assures us that the row area is a set of distinct numbers in $\{1,\ldots, n-2\}$. So to each path of height $n-d-2$ there is a tableau with a descent set corresponding to adding one to each element of the set of row area and appending the element $1$. Hence, $\beta_d$ is indeed a bijection.

 Let $n_i$ be the number of east steps before the $i$-th north step of $\beta(\tau)_d$. Then $n_i=n-r_{n-d-i}-i$ by definition of the map $\beta_d$. The area of a path is equal to $\sum_{i=1}^{\h(\beta_d(\tau))}n-i-1-n_i$ therefore $\A(\beta_d(\tau))=\sum_{i=1}^{\h(\beta_d(\tau))}(r_{n-d-i}-1)=(\m(\tau)-1)-\h(\beta_d(\tau))$.
\end{proof}
We can now improve the aspect of our formula.
\begin{prop}
For $R=\{(a,2,1^k)~|~ k\in\mathbb{N}, a\in \mathbb{N}_{\geq 2}\}$, we have:
\begin{equation*}\langle \mathcal{E}_{n,n},e_n\rangle|_{\mathcal{R}}=\sum_{i=2}^{n-3}\sum s_{\shape T(\gamma)},
\end{equation*}
where the second sum is over paths $\gamma\in T_{n,0}$ such that $\gamma\not=N\tilde{\gamma}N^{i-1}$ and $i\leq \h(\gamma)\leq n-3$. Additionally, $\shape T(\gamma)$ is the partition given by $\left(\A(\gamma)+\h(\gamma)+1-i,2,1^{n-3-\h(\gamma)}\right)$.
\end{prop}

\begin{proof}For a fixed $k$ and $i$ the sum in Equation \eqref{Eq : 2 colone lift} is over all tableaux in $\SYT(k+1,1^{n-k-1})_{\{1\}}$ for which the descent set doesn't contain the subset $\{1,2,\ldots, i, n-1\}$. Using the bijection in Lemma \ref{Lem : bij descent} we obtain a sum over all paths in $T_{n,0,n-k-2}$ that can not be written as $N\tilde{\gamma}N^{i-1}$.

For $\gamma \in T_{n,0,n-k-2}$ $n-3-\h(\gamma)=k-1$ and $\m(\tau)-i=\A(\beta_k(\tau))+\h(\beta_k(\tau))+1-i$ we have:

\begin{align*}\langle \mathcal{E}_{n,n},e_n\rangle|_{\mathcal{R}}=&\sum_{k=1}^{n-4}\sum_{i=2}^{n-k-2}\underset{\gamma\not=N\tilde{\gamma}N^{i-1}}{\sum_{\gamma\in T_{n,0,n-k-2}}} s_{\shape T(\gamma)}
\\			&=\sum_{i=2}^{n-3}\sum_{k=1}^{n-i-2}\underset{\gamma\not=N\tilde{\gamma}N^{i-1}}{\sum_{\gamma\in T_{n,0,n-k-2}}} s_{\shape T(\gamma)}
\\			&=\sum_{i=2}^{n-3}\underset{i\leq \h(\gamma)\leq n-3}{\underset{\gamma\not=N\tilde{\gamma}N^{i-1}}{\sum_{\gamma\in T_{n,0}}}} s_{\shape T(\gamma)}
\end{align*}
\end{proof}

We have not found a way to show that $e^\perp_1(\langle \mathcal{E}_{n,n},e_n\rangle|_{\mathcal{ R}})|_{\hooks}$ is equivalent to Equation \eqref{Eq : difference}. Such an equivalence would prove the case $\mu=2,1^{n-2}$.


\section{Conclusion and further questions}

It would be interesting to show that the formulas hold for all $\mu$ when $r=1$ and all $r$ when $\mu=1^n$. As we mentioned previously, this could be obtained, for $\mu$ a hook, by having a formula for the restriction to Schur functions indexed by partitions
with two columns. If one could write the $q,t$ statistics of the $m$-Schr\"oder paths in terms of Schur functions we could have a more general formula restricting only to shape $(a,b,1^j)$ with $a$ and $b$ arbitrary. Finally, the author is already working on showing how existing formulas for special cases are equivalent to the proposed formula. 

\section*{acknowledgments} 
I would like to thank  Fran\c cois Bergeron for sharing a draft of \cite{[B2018]} and for related useful discussions. Thank you to Franco Saliola for proof reading and useful comments. This is the extended version of the FPSAC 2019 contribution \cite{[Wal2019d]}.

\bibliographystyle{alpha}
\bibliography{bibliographie}

\end{document}